\documentclass[a4paper, 12pt]{article} 
\usepackage{amsmath} 
\usepackage{amsthm}
\usepackage[all]{xy}
\usepackage{amsfonts}
\usepackage{amssymb}
\usepackage{graphicx}      
\input{diagxy}

\makeatletter

\newcommand{\pocorner}{\hbox to 10pt{{\vrule height10pt depth0pt width0.5pt}%
    \vbox to 10pt{{\hrule height0.5pt width9.5pt depth0pt}\vfill}}}
\newcommand{\@poexcursion}[1]{\save[]-<0pt+#1,-12pt>*{\pocorner}\restore}
\newcommand{\pbcorner}{\vbox to 0pt{\kern 5pt\hbox to 0pt{\kern 5pt%
      \vbox{{\hrule height0.5pt width9.5pt depth0pt}}%
      {\vrule height10pt depth0pt width0.5pt}\hss}\vss}}
\def\@pbexcursion[#1]{\save[]+DR-<16pt,-8pt>+<#1,0pt>*{\pbcorner}\restore}
\newcommand{\pbexcursion}{\@ifnextchar[\@pbexcursion{\@pbexcursion[0pt]}}

\def\arr@fn{\futurelet\arr@next}
\def\arr@dn{\def\arr@next}

\newtoks\arr@toks
\def\addtoarr@toks#1{\arr@toks=\expandafter{\the\arr@toks#1}}

\def\arr@upperlab{}
\def\arr@lowerlab{}
\def\arr@mods{}
\def\arr@uppermods{}
\def\arr@lowermods{}

\newbox\arr@box
\newdimen\arr@dimen
\newdimen\arr@spacer
\newif\ifarr@fixeddim

\def\arr@updatedimen#1{%
  \ifarr@fixeddim\else\setbox\arr@box=\hbox{$\m@th\scriptstyle #1$}%
  \ifdim\wd\arr@box>\arr@dimen \arr@dimen=\wd\arr@box\fi\fi}

\def\parsearr@@{%
    \ifx\space@\arr@next \expandafter\arr@dn\space{\arr@fn\parsearr@@}%
    \else\ifx ^\arr@next \arr@dn ^##1{\def\arr@upperlab{^{##1}}\arr@fn\parsearr@@}%
    \else\ifx _\arr@next \arr@dn _##1{\def\arr@lowerlab{_{##1}}\arr@fn\parsearr@@}%
    \else\ifx <\arr@next \arr@dn <##1>{\arr@fixeddimtrue\arr@dimen=##1%
      \arr@fn\parsearr@@}
    \else\addAT@\ifx\arr@next \addAT@\arr@dn[##1]{\def\arr@mods{##1}\arr@fn\parsearr@@}%
    \else\ifarr@fixeddim\else\arr@updatedimen{\arr@upperlab}%
        \arr@updatedimen{\arr@lowerlab}\advance\arr@dimen by \arr@spacer\fi%
      \arr@toks={\xy<0pt,0pt>*+{}\ar}%
      \expandafter\addtoarr@toks\expandafter{\arr@mods}%
      \expandafter\addtoarr@toks\expandafter{\arr@upperlab}%
      \expandafter\addtoarr@toks\expandafter{\arr@lowerlab}%
      \addtoarr@toks{<\arr@dimen,0pt>*+{}\endxy}%
      \arr@dn{\the\arr@toks}%
    \fi\fi\fi\fi\fi\arr@next}

\newcommand{\arrow}{%
  \arr@fixeddimfalse\arr@dimen=0pt\def\arr@upperlab{}\def\arr@lowerlab{}%
  \def\arr@mods{}\arr@fn\parsearr@@}

\def\parsearrpair@@{%
    \ifx\space@\arr@next \expandafter\arr@dn\space{\arr@fn\parsearrpair@@}%
    \else\ifx ^\arr@next \arr@dn ^##1{\def\arr@upperlab{^{##1}}\arr@fn\parsearrpair@@}%
    \else\ifx _\arr@next \arr@dn _##1{\def\arr@lowerlab{_{##1}}\arr@fn\parsearrpair@@}%
    \else\ifx <\arr@next \arr@dn <##1>{\arr@fixeddimtrue\arr@dimen=##1%
      \arr@fn\parsearrpair@@}
    \else\addAT@\ifx\arr@next \addAT@\arr@dn[{\arr@fn\parsearrpair@@@}%
    \else\ifarr@fixeddim\else\arr@updatedimen{\arr@upperlab}%
        \arr@updatedimen{\arr@lowerlab}\advance\arr@dimen by \arr@spacer\fi%
      \arr@toks={\xy<0pt,0pt>*+{}="a", <\arr@dimen,0pt>*+{}="b"}%
      \expandafter\addtoarr@toks\expandafter{\addAT@\ar<2pt>}%
      \expandafter\addtoarr@toks\expandafter{\arr@uppermods}%
      \addtoarr@toks{"a";"b"}%
      \expandafter\addtoarr@toks\expandafter{\arr@upperlab}%
      \expandafter\addtoarr@toks\expandafter{\addAT@\ar<-2pt>}%
      \expandafter\addtoarr@toks\expandafter{\arr@lowermods}%
      \addtoarr@toks{"a";"b"}%
      \expandafter\addtoarr@toks\expandafter{\arr@lowerlab}%
      \addtoarr@toks{\endxy}\arr@dn{\the\arr@toks}%
    \fi\fi\fi\fi\fi\arr@next}

\def\parsearrpair@@@{%
  \ifx u\arr@next \arr@dn u##1]{\def\arr@uppermods{##1}\arr@fn\parsearrpair@@}%
  \else\arr@dn l##1]{\def\arr@lowermods{##1}\arr@fn\parsearrpair@@}%
  \fi\arr@next}

\newcommand{\arrowpair}{%
  \arr@fixeddimfalse\arr@dimen=0pt\def\arr@upperlab{}\def\arr@lowerlab{}%
  \def\arr@uppermods{}\def\arr@lowermods{}\arr@fn\parsearrpair@@}

\makeatother

\newdir{ >}{{}*!/-10pt/@{>}}
\newdir{u(}{{}*!/-6pt/@^{(}}
\newdir{d(}{{}*!/-6pt/@_{(}}
\newdir{|>}{%
  !/4.5pt/@{|}*:(1,-.2)@^{>}*:(1,+.2)@_{>}*+@{}}

\newcommand{\epi}{\arrow@[@{->>}]}
\newcommand{\cover}{\arrow@[@{-|>}]}
\newcommand{\spanarr}{\arrow@[|-@{|}]}
\newcommand{\inc}{\arrow@[@{u(->}]}
\newcommand{\mapto}{\arrow@[@{|->}]}
\newcommand{\mono}{\arrow@[@{ >->}]}
\newcommand{\mat}{\arrow@[|-*{\object@{o}}]}
\newcommand{\twocell}{\arrow@[@{=>}]}

\newbox\linebox \setbox\linebox=\hbox{\xy \POS(0,40)\ar@{-} (0,0)\endxy}
\def\line{\copy\linebox}
\newbox\dlinebox \setbox\dlinebox=\hbox{\xy \POS(-32,32)\ar@{-} (0,0)\endxy}
\def\dline{\copy\dlinebox} 
  
\title{A Category of quantum categories}
\author{Dimitri Chikhladze}
\date{}

\begin{document}
\theoremstyle{plain}
\newtheorem{theorem}{Theorem}
\newtheorem{proposition}[theorem]{Proposition} 
\newtheorem{lemma}[theorem]{Lemma} 
\newtheorem{statement}[theorem]{Statement}
\theoremstyle{definition}
\newtheorem{definition}[theorem]{Definition}
\theoremstyle{remark}
\newtheorem{example}[theorem]{Example} 
     
\maketitle
\begin{abstract}
Quantum categories were introduced in \cite{quantumcat} as generalizations of
both bi(co)algebroids and small categories. We clarify details of that work. In
particular, we show explicitly how the monadic definition of a quantum category
unpacks to a set of axioms close to the definitions of a bialgebroid in the Hopf
algebraic literature. We define notions of functor and natural transformation
for quantum categories.
\end{abstract}

\section{Introduction}  

Quantum categories are defined within a monoidal category satisfying
a modest condition. When the monoidal category is the opposite category of modules over a
commutative ring, a quantum category is the same as what has been called a
bialgebroid in the literature. Bialgebroids, under the name of
$\times_A$-bialgebras, appeared as early as M. Takeuchi's paper \cite{Takeuchi}.
Another push came later when, essentially the same notion, independently from Takeuchi's work and for
totally different reasons, was considered by J.-H. Lu \cite{Lu} and P. Xu
\cite{Xu}. A categorical approach was introduced in a paper by B. Day and R.
Street \cite{quantumcat}. There, a quantum category in a general monoidal
category was defined, so incorporating both bialgebroids, in the way mentioned
above, and ordinary categories, by taking the monoidal category to be the
category of sets.

In this paper, influenced by \cite{quantumcat}, we approach quantum
categories using the bicategorical version of the formal theory of (co)monads
\cite{FTM1}. We give a monadic definition of a quantum category and show
explicitly how it translates to a set of axioms close to \cite{Lu} and
\cite{Xu}. We define the notion of functor between quantum categories, obtaining a category $\mathrm{qCat}$. 

The paper is organised in the following way. First, in Section 2 we review the
formal theory of (co)monads in a bicategory. Then in Sections 3 and 4,  we deal
with a particular bicategory, which is defined from our base monoidal category.
Sections 5 and 6 are dedicated to quantum categories and related concepts. In
Section 7  applications are given, which show that some constructions
which otherwise would be complicated become simple when the categorical
approach is taken. In the Appendix we introduce framed string diagrams designed
to ease computations involving the quantum structures.

\section{Monoidal comonads}  

Let $\mathcal{B}$ be a bicategory. We write as if $\mathcal{B}$
were a 2-category, regarding associativity and unitivity isomorphisms as
identities. 
 
Recall that a comonad in $\mathcal{B}$ \cite{FTM1}, \cite{B} is a pair $(B, g)$,
where $B$ is an object of $\mathcal{B}$ and $g = (g, \delta : g \Rightarrow gg,
\epsilon : g \Rightarrow 1_g)$ is a comonoid in the homcategory $\mathcal{B}(B,
B)$. A map of comonads $(k, \kappa) : (B, g) \rightarrow (A, g')$ consists of a morphism $k : B \rightarrow
A$ and a 2-cell $\kappa : kg \Rightarrow g'k$ satisfying: 

$$\Big(kg \to^{k\delta} kgg \to^{\kappa g} g'kg \to^{g'\kappa} g'g'k\Big) \; = \;
\Big(kg \to^{\kappa} g'k \to^{g'\delta} g'g'k\Big)\text{,}$$

$$\Big(kg \to^{k\epsilon} k\Big) = \Big(kg \to^{\kappa} g'k \to^{\epsilon k}
k\Big)\text{.}$$

\noindent A comonad map transformation $\tau : (k, \kappa) \Rightarrow (k',
\kappa') : (A, g) \to (B, g')$ is a 2-cell $\tau : k \Rightarrow k'$ satisfying:

$$\Big(kg \to^{\tau g} k'g \to^{\kappa'} g'k'\Big) = \Big(kg \to^{\kappa} g'k
\to^{g\tau} g'k'\Big)\text{.}$$

\noindent Comonads in $\mathcal{B}$, comonad maps and comonad map
transformations form a bicategory $\mathrm{Comnd}\mathcal{B}$ under the obvious
composition.

$\mathcal{B}$ is said \cite{FTM1} to
admit the Eilenberg-Moore construction for comonads if the inclusion $\mathcal{B}
\rightarrow \mathrm{Comnd}\mathcal{B}$, taking an object $B$ to $(B, 1)$, has a
right biadjoint $\mathrm{Comnd}\mathcal{B} \rightarrow \mathcal{B}$. The value of
this right biadjoint at $(B, g)$ is called an
Eilenberg-Moore object of $(B, g)$. It will be denoted by denoted $B^g$. There
is a pseudonatural equivalence

$$\mathcal{B}(X, B^g) \simeq \mathrm{Comnd}\mathcal{B}((X, 1), (B, g))$$

\noindent The objects of the right side are called $g$-coalgebras. Taking $X =
B^g$ and evaluating at the identity, we obtain a universal $g$-coalgebra $(u, \gamma)
: (B^g, 1) \rightarrow (B, g)$. Every comonad map $k : (B, g) \rightarrow (A,
g')$ induces a map $\hat{k} : B^g \rightarrow A^{g'}$ between Eilenberg-Moore
objects so that there is an isomorphism:

\begin{equation}
\bfig
\square[B^g`A^{g'}`B`A;\hat{k}`u`u`k]
\place(250,250)[\cong]
\efig
\label{lift}
\end{equation} 

\noindent By an equivalence between suitable
categories, comonad structures on $k : B \rightarrow A$ correspond to diagrams
\eqref{lift} in $\mathcal{B}$.

Let $\mathcal{B}$ be a monoidal bicategory \cite{MBHA}.
We specify $n$-ary tensor product pseudofunctors

$$\mathcal{B}^n \to^{\otimes_n} \mathcal{B}$$

\noindent by choosing bracketing for the tensor
product to be from the left. So, the expression $B_1\otimes\ldots\otimes B_n$
refers to $\otimes_n(B_1\otimes\ldots\otimes B_n)$.

A monoidale $E$ in $\mathcal{B}$ consists of an object $E$ together
with morphisms $p : E\otimes E \rightarrow E$ and $j : I \rightarrow E$ called
the multiplication and the unit respectively, and invertible 2-cells expressing associativity and
unitivity, subject to coherence conditions. The $n$-ary multiplication map

$$E^n \to^{p_n} E\text{.}$$

\noindent is defined by consecutive multiplications from the left.

\begin{example}
Let an object $B$ be the right bidual to an object $A$ in $\mathcal{B}$,
with the biduality counit $e : A\otimes B \rightarrow I$ and the biduality
unit $n: I \rightarrow B\otimes A$. $B\otimes A$ becomes a monoidale with product $p =
1\otimes e \otimes 1 : B\otimes A\otimes B\otimes A \rightarrow B \otimes A$
and unit $j = n : I \rightarrow B\otimes A$.
\end{example}

A monoidal morphism $(f, \phi_2, \phi_0) : E \rightarrow D$ between monoidales
consists of a morphism $f : E \rightarrow D$ and 2-cells $\phi_2 : p(f\otimes f) \Rightarrow fp$, $\phi_0 : j
\Rightarrow fj$ satisfying three axioms. The composition of monoidal morphisms
$(f, \phi_2, \phi_0) : E \rightarrow D$ and $(f', \phi_2, \phi_0) : D \rightarrow F$ is defined to be $(f'f, \phi_2, \phi_0)
: E \rightarrow F$, where

$$\phi_2 = \Big(p(f'\otimes f')(f\otimes f) \to^{\phi_2(f\otimes f)}
f'p(f\otimes f) \to^{f'\phi_2} f'fp\Big)$$
$$\phi_0 = \Big(j \to^{\phi_0} f'j \to^{f'\phi_0} f'fj\Big)\text{.}$$

\noindent A monoidal morphism is called strong when $\phi_2$
and $\phi_0$ are isomorphisms. Monoidales in $\mathcal{B}$, monoidal morphisms
between them and obvious 2-cells form a bicategory $\mathrm{Mon}\mathcal{B}$.

There is a biequivalence

\begin{equation}
\mathrm{MonComnd}\mathcal{B} \sim \mathrm{ComndMon}\mathcal{B}\text{,}
\label{bieq}
\end{equation}

\noindent where the left hand side is defined using the monoidal structure on
$\mathrm{Comnd}\mathcal{B}$ inherited from $\mathcal{B}$.

A monoidal comonad is an object of $\mathrm{ComndMon}\mathcal{B}$, or
equally, an object of $\mathrm{MonComnd}\mathcal{B}$. Explicitely, a
monoidal comonad consists of a monoidale $E$, a comonad $g$ on
$E$ and 2-cells $\phi_2 : p(g\otimes g) \Rightarrow gp$, $\phi_0 : j \Rightarrow
gj$ such that $(g, \phi_2, \phi_0)$ is a monoidal morphism and
$(p, \phi_2) : (E\otimes E, g\otimes g) \rightarrow (E, g)$ and $(j, \phi_0) : (I, 1) \rightarrow (E, g)$ are
comonad maps. A morphism of monoidal comonads $(k, \kappa) : (E, g) \rightarrow
(E', g')$ is a map of underlying comonads such that $\kappa : kg \Rightarrow g'k$
is a map of monoidal morphisms.
 
$\mathrm{Mon}(-)$ can be made into a pseudofunctor from
the tricategory of monoidal bicategories and monoidal pseudofunctors to the
tricategory of bicategories and pseudofunctors. Since the inclusion $i :
\mathcal{B} \rightarrow \mathrm{Comnd}\mathcal{B}$ is a strong monoidal pseudofunctor the
right biadjoint to it is a monoidal pseudofunctor too. It follows that if $i$
has a right biadjoint, then $\mathrm{Mon}(i) : \mathrm{Mon}\mathcal{B} \rightarrow
\mathrm{MonComnd}\mathcal{B}$ has a right biadjoint too. Using the the
biequivalence \eqref{bieq} we infer that the canonical inclusion
$\mathrm{Mon}\mathcal{B} \rightarrow \mathrm{ComndMon}\mathcal{B}$ has a right
biadjoint. This proves \cite{IM}, \cite{PM}:

\begin{proposition}\label{monem}
If $\mathcal{B}$ admits the Eilenberg-Moore construction for comonads, then so
does $\mathrm{Mon}(\mathcal{B})$.

\end{proposition}

Explicitly an Eilenberg-Moore object of a monoidal comonad $(E, g)$ is obtained
in the following way. Let $E^g$ be the Eilenberg-Moore object for the
underliying comonad in $\mathcal{B}$ with $(u, \gamma) : (E^g, 1) \rightarrow (E, g)$ the universal coalgebra. Then $p(u\otimes u) : E^g\otimes E^g
\rightarrow E$ becomes a $g$-coalgebra with coaction

$$p(u\otimes u) \to^{p(\gamma\otimes\gamma)} p(g\otimes g)(u\otimes u)
\to^{\phi_2p} gp(u\otimes u)\text{,}$$

\noindent and $j : I \rightarrow E$ becomes a $g$-coalgebra with the coaction 

$$j \to^{\phi_0} gj\text{.}$$

\noindent The induced morphisms $\hat{p} : E^g\otimes E^g \rightarrow E^g$ and
$\hat{j} : I \rightarrow E^g$ define a monoidale structure on $E^g$. This monoidale is the
Eilenberg-Moore object of $(E, g)$ in $\mathrm{Mon}\mathcal{B}$. Moreover, the
map $u: E^g \rightarrow E$ is a strong monoidal morphism.

There is an equivalence of categories which establishes a correspondance between
monoidal comonad maps $(k, \kappa) : (E, g) \rightarrow (E', g')$ and diagrams
\eqref{lift}, now in $\mathrm{Mon}\mathcal{B}$.

What we have been discussing so far were standard constructions in a monoidal
bicategory. Further we introduce some concepts, which we will later use for our
specific purposes.
 
An opmonoidal morphism $(w, \psi_2, \psi_0) : E \rightarrow D$
between monoidales is a monoidal morphism in $\mathcal{B}^{co}$. Thus an
opmonoidal morphism consists of a morphism $w : E \rightarrow F$ and
2-cells $\psi_2 : wp \Rightarrow p(w\otimes w)$, $\psi_0 : hj \Rightarrow j$ in $\mathcal{B}$ satisfying three axioms.

Monoidal morphisms and opmonoidal morphisms lead
us to the setting of a double category \cite{KS}, \cite{V}. Recall briefly, that
a double category has objects and two types of arrows, called horizontal morphisms and vertical morphisms, forming bicategories in the two directions.
Also, there is a set of squares, each square having as its sides two
horizontal morphisms and two vertical morphisms. Squares can be composed in
the two directions.

As suggested, there is a double category with objects the monoidales
in $\mathcal{B}$, horizontal arrows the monoidal morphisms and vertical morphisms
the opmonoidal morphisms. A square is a 2-cell

$$\bfig
\square[E`E'`D`D';f`w`w'`f']
\morphism(250,300)/=>/<0,-100>[`;\sigma]
\efig$$ 

\noindent with $f$ and $f'$ monoidal morphisms and $w$ and $w'$ opmonoidal
morphisms such that:

$$\Big( w'p(f\otimes f) \to^{\psi_2(f\otimes f)} p(w'\otimes w')(f\otimes f)
\to^{p(\sigma\otimes\sigma)} p(f'\otimes f')(w\otimes w) \to^{\phi_2(w\otimes
w)} f'p(w\otimes w)\Big)$$ 
$$= \Big( w'p(f\otimes f) \to^{w'\phi_2} w'fp \to^{\psi_2p}
f'wp \to^{\sigma} f'p(w\otimes w)\Big)$$
$$\text{and} \qquad \Big( w'j \to^{\phi_0} w'fj \to^{\sigma j} f'wj
\to^{f'\psi_0} f'j) = (w'j \to^{psi_0} j \to^{\phi_0} f'j \Big)\text{.}$$

Suppose that $(E, g)$ and $(D, g')$ are monoidal comonads. An opmorphism of
monoidal comonads $(h, \sigma) : (E, g) \rightarrow (D, g')$ is an opmonoidal
morphism $h : E \rightarrow D$ together with a square $\sigma : hg \Rightarrow
g'h$, such that $(h, \sigma) : (E, g) \rightarrow (D, g')$ is a map of comonads.

As with the monoidal comonad maps, there is an equivalence of categories
which establishes a correspondence between opmorphisms of monoidal comonads $h :
(E, g) \rightarrow (D, g')$ and diagrams

$$
\bfig
\square[E^g`D^{g'}`E`D;\hat{h}`u`u`h]
\place(250,250)[\cong]
\efig
$$ 

\noindent of opmonoidal morphisms. 

A coaction of an opmonoidal morphism $h : E \rightarrow D$ on a morphism $l : E
\rightarrow D$ of $\mathcal{B}$ is a 2-cell $\lambda : lp \Rightarrow
p(h\otimes l)$ satisfying two axioms, relating it to the opmonoidal structure on
$h$.

Suppose that $(h, \sigma) : (E, g) \to (D, g')$ is an opmorphism
of monoidal comonads and $(l, \tau) : (E, g) \to (D, g')$ is a comonad map.
We will say that a left coaction $\lambda$ of $h$ on $l$ respects the
comonad structure if

$$\Big( lp(g\otimes g) \to^{\lambda(g\otimes g)} p(h\otimes l)(g\otimes g)
\to^{p(\sigma\otimes\tau)} p(g'\otimes g')(h\otimes l) \to^{\mu(h\otimes l)}
g'p(h\otimes l)\Big)$$ 
$$= \Big( lp(g\otimes g) \to^{l\mu} lgp \to^{\tau p} g'lp \to^{g'\lambda}
g'p(h\otimes l)\Big)\text{.}$$

A left coaction of $h$ on
$l$ respects comonad structure if and only if it can be lifted to a coaction of
$\hat{l} : E^g \rightarrow D^{g'}$ on $\hat{h} : E^g \rightarrow D^{g'}$.

There is a similar notion of a right coaction of an opmorphism. 

\section{The bicategory of comodules}\label{bc}

Suppose that $\mathcal{V} = (\mathcal{V}, \otimes, I,
c)$ is a braided monoidal category with finite colimits. Assume that each of the functors
$X\otimes -$ preserves equalizers of coreflexive pairs.

We will work with monoidal bicategory $\mathcal{C} =
\mathrm{Comod}\mathcal{V}$ defined in \cite{DMS}. Objects of $\mathcal{C}$ are
the comonoids

$$C = (C, \delta : C \rightarrow C\otimes C, \epsilon : C \rightarrow I)$$

\noindent in $\mathcal{V}$. The homcategory $\mathcal{C}(C, D)$ is the category
of Eilenberg-Moore coalgebras for the comonad $C\otimes - \otimes D : \mathcal{V}
\rightarrow \mathcal{V}$. A 1-cell from $C$ to $D$, depicted $C \spanarr D$, is a
comodule from $C$ to $D$. Recall that this consists of an object $M$ and a coaction
map $\delta : M \rightarrow C\otimes M \otimes D$ satisfing two axioms. A 2-cell $\alpha : M
\Rightarrow N : C \spanarr D$ is a coaction respecting map $M \rightarrow N$.
An object of $\mathcal{C}(C, I)$ is a left $C$-comodule, and an object of
$\mathcal{C}(I, C)$ is a right $C$-comodule. A comodule $M : C \spanarr D$
becomes a left $C$-comodule and a right $D$-comodule via coactions

\begin{align}
\delta_l : M \to^{\delta} C\otimes M\otimes D \to^{1\otimes 1 \otimes \epsilon} C\otimes M \nonumber\\
\label{lr}\\
\nonumber\delta_r : M \to^{\delta} C\otimes M\otimes D\to^{\epsilon\otimes 1 \otimes 1} M\otimes D\text{.}
\end{align}

\noindent The maps $\delta_l$ and $\delta_r$ are called left and right coactions
on $M$. If $M$ is a left $C$-comodule and $N$ is a right
$D$-comodule, then a tensor product $M\otimes_C N$ over $C$ is defined by a
(coreflexive) equalizer:

$$M\otimes_C N\to^i M\otimes N \two^{\delta_r\otimes 1}_{1\otimes \delta_l}
M\otimes C\otimes N\text{.}$$

If $M$ is a comodule $E \spanarr C$ and $N$ is a comodule $C \spanarr F$, then
using the fact that the functor $E\otimes - \otimes F$ preserves coreflexive
equalizers, $M\otimes_C N$ becomes a comodule $E \spanarr F$. Composition in
$\mathcal{C}$ is defined by $N\circ M = M\otimes_C N$. It is associative up to
canonical isomorphism.

Any comonoid $C$ is a $C \spanarr C$ comodule with the coaction

$$C \to^{\delta_3} C\otimes C\otimes C\text{.}$$

\noindent The identity comodule on $C$ is $C$ itself.

As it is a convention to name such bicategories after arrows,
$\mathrm{Comod}{\mathcal{V}}$ is called the bicategory of comodules. For more on
the theory of comodules we refer the reader to \cite{path}.

Each comonoid morphism $f: C \rightarrow D$ determines
an adjoint pair in $\mathcal{C}$:

$$\bfig
\morphism/@{>}|-*@{|}/[f_\ast\dashv f^\ast : C`D;]
\efig
$$ 

\noindent The comodules $f^\ast : C \spanarr D$ and $f_\ast : D \spanarr C$ are
both $C$ as objects of $\mathcal{V}$ with coactions respectively

$$C \to^{\delta_3} C\otimes C\otimes C \to^{1\otimes 1\otimes f} C\otimes C\otimes D \quad\text{and}$$

$$C \to^{\delta_3} C\otimes C\otimes C \to^{f\otimes 1\otimes 1} D\otimes D\otimes C\text{.}$$

\noindent The counit 

$$
\bfig
\Atriangle/@{<-}|-*@{|}`@{>}|-*@{|}`@{>}|-*@{|}/<400,400>[C`D`D;f^\ast`f_\ast`D]
\morphism(400,200)|r|/=>/<0,-100>[`;\beta]
\efig
$$ 

\noindent of the adjunction is the map

$$C\otimes_CC \cong C \to^f D\text{.}$$

\noindent The unit

$$
\bfig
\Atriangle/@{<-}|-*@{|}`@{>}|-*@{|}`@{>}|-*@{|}/<400,400>[D`C`C;f_\ast`f^\ast`C]
\morphism(400,100)/=>/<0,100>[`;\alpha]
\efig
$$
 
\noindent is induced by the comultiplication $\delta : C \rightarrow C\otimes C$
as shown on the diagram:

$$
\bfig
\Vtriangle/>`<-`<-/<300,300>[C\otimes_DC`C\otimes C`C;eq.`\alpha`\delta]
\place(1550,300)[C\otimes D\otimes C]
\morphism(750,320)<550,0>[`;(1\otimes f\otimes1)(\delta\otimes1)]
\morphism(750,280)|b|<550,0>[`;(1\otimes f\otimes1)(1\otimes\delta)]
\efig 
$$

\noindent In particular, we have comodules $\epsilon^\ast : I \spanarr C$ and
$\epsilon_\ast : C \spanarr I$. The compositions

$$C \spanarr^M D \spanarr^{\epsilon_\ast} I$$

$$I \spanarr^{\epsilon_\ast} C \spanarr^M D$$

\noindent reconfirm the fact that $M$ is a left $C$-comodule and a right
$D$-comodule by \eqref{lr}.

The monoidal structure on $\mathcal{C}$ extends the
monoidal structure on $\mathcal{V}$. The tensor product of comonoids $C = (C, \delta,
\epsilon)$ and $C' = (C', \delta', \epsilon')$ is $C\otimes C'$ with comultiplication and counit:

$$(1\otimes c\otimes 1)(\delta\otimes\delta') : C\otimes C' \rightarrow C\otimes C'\otimes C\otimes C'$$

$$\epsilon\otimes\epsilon' : C\otimes C' \rightarrow I\text{.}$$
 
\noindent The monoidal unit of $\mathcal{C}$ is $I$, which is a comonoid in an
obvious way. On 1-cells, the tensor product of comodules $M : C \spanarr D$ and $N : C'
\spanarr D'$ is $M\otimes N$, which is a comodule $C\otimes C' \spanarr D\otimes
D'$ with coaction:
 
$$M\otimes N \to^{\delta\otimes\delta} C\otimes M\otimes D\otimes C'\otimes N\otimes D' \to^{c_{142536}} C\otimes C'\otimes M\otimes N\otimes D\otimes D'\text{.}$$

\noindent Here and below a morphism named $c$ subscripted with a permutation is
an isomorphism coming from the braiding.

We often encounter comodules going between tensor products of comonoids,
like $M : C_1\otimes C_2\otimes \ldots C_n \spanarr D_1\otimes D_2\otimes \ldots
D_m$. Such a comodule inherently is a left $C_i$-comodule, for $1 \leq i \leq
n$, and a right $D_i$-comodule, for $1 \leq i \leq m$.
Conversely, given left $C_i$-comodule and right $D_i$-comodules structures on
$M$ compatible in a certain way, $M$ becomes a $C_1\otimes C_2\otimes
\ldots C_n \spanarr D_1\otimes D_2\otimes \ldots D_m$ comodule. This enables us
to describe a comodule just by giving left and right coactions. A map $M
\rightarrow N$ is a comodule map between comodule $M, N : C_1\otimes C_2\otimes
\ldots C_n \spanarr D_1\otimes D_2\otimes \ldots D_m$ if and only if it is a left $C_i$-comodule
map for all $1 \leq i \leq n$ and a right $D_i$-comodule map for all $0 \leq i \leq m$.

$\mathcal{C}$ is a right autonomous monoidal
bicategory. The bidual of a comonoid $C = (C,\delta, \epsilon)$ is the comonoid with the opposite
comultiplication $C^{o} = (C, c\delta, \epsilon)$. Unit and counit are comodules
$e : C^o\otimes C \spanarr I$ and $n : I \spanarr C\otimes C^o$, both of which
are $C$ as objects of $\mathcal{V}$ and the coactions on them are respectively

$$C \stackrel{\delta_3}{\rightarrow} C\otimes C\otimes C \stackrel{1\otimes c}{\rightarrow} C\otimes C\otimes C\quad\text{and}$$

$$C \stackrel{\delta_3}{\rightarrow} C\otimes C\otimes C \stackrel{c\otimes 1}{\rightarrow} C\otimes C\otimes C\text{.}$$

It follows that $C^o\otimes C$ is a monoidale in $\mathcal{C}$. The multiplication is $p = 1\otimes e\otimes
1$ and the unit is $j = n$. Still more explicitly, the multiplication
$C^o\otimes C\otimes C^o\otimes C \spanarr C^o\otimes C$ is $p = C\otimes
C\otimes C$ with coaction

$$C^{\otimes 3} \to^{\delta_3\otimes\delta_3\otimes\delta_3} C^{\otimes 9} \to^{c_{146725839}} C^{\otimes 9}$$

\noindent and the unit $I \spanarr C^o\otimes C$ is $j = C$ with coaction

$$C \to^{\delta_3} C^{\otimes 3} \to^{c_{213}} C^{\otimes 3}$$

Let $M$ and $N$ be comodules $I \spanarr C^o\otimes C$. Regard these as
comodules $C \spanarr C$ by the equivalence

\begin{equation}
\mathcal{C}(I, C^o\otimes C) \simeq \mathcal{C}(C, C)\text{.}
\label{CC}
\end{equation}

\noindent The composite

$$I \spanarr^{M\otimes N} C^o\otimes C\otimes C^o\otimes C \spanarr^p C^o\otimes
C$$

\noindent is $M\otimes_C N$ with right $C^o\otimes C$-coaction the
unique map $\delta_l : M\otimes_C N \to (M\otimes_C N)\otimes C\otimes C$ making

$$
\bfig
\square<1200,500>[M\otimes_C N`M\otimes N`(M\otimes_CN)\otimes C\otimes
C`M\otimes N\otimes C\otimes C;i`\delta_l`\delta_l`i\otimes C\otimes C]
\efig
$$

\noindent commute. 

The equivalence \eqref{CC} is a monoidal equivalence,
where the monoidal stucture on the left side comes from the pseudomonoid structure on $C^o\otimes C$ and
the monoidal structure on the right is defined to be the composition in
$\mathcal{C}$.

Next we prove some technical lemmas, which we use in Section
\ref{secqc}.

\begin{lemma}\label{l1} Suppose that $\beta$ is a 2-cell:

$$
\bfig
\Atriangle/@{>}|-*@{|}`@{>}|-*@{|}`@{>}|-*@{|}/<500,500>[A\otimes C\otimes
C^o\otimes B`A\otimes B`D;A\otimes e\otimes B`M`N]
\morphism(500,250)/=>/<0,-100>[`;\beta]
\efig
$$ 

\noindent Let $\alpha : M \to N$ be the map in $\mathcal{V}$ determined by the
pasting composite

\begin{equation}
\bfig
\square/@{>}|-*@{|}`@{>}|-*@{|}``@{>}|-*@{|}/<850,600>[A\otimes B`A\otimes
C\otimes C^o\otimes B`A\otimes B`A\otimes
B;A\otimes\epsilon^\ast\otimes B`A\otimes B``A\otimes B]
\btriangle(850,0)/@{>}|-*@{|}`@{>}|-*@{|}`@{>}|-*@{|}/<600,600>[A\otimes
C\otimes C^o\otimes B`A\otimes B`D;A\otimes e\otimes B`M`N]
\morphism(1050,300)/=>/<0,-100>[`;\beta] 
\morphism(425,350)/=>/<0,-100>[`;A\otimes\epsilon \otimes B]
\efig
\label{pc1}
\end{equation}

\noindent It satisfies

\begin{equation}
M \to^{\delta_l} C\otimes C\otimes M \two^{1\otimes\epsilon\otimes 1}_{\epsilon\otimes 1\otimes 1} C\otimes M \to^{1\otimes\alpha} C\otimes N
\label{1}
\end{equation}

\noindent The 2-cell $\beta$ is uniquely determined by a left $A\otimes B$-
right $D$-comodule map $\alpha$ which satisfies \eqref{1}.

\end{lemma} 

\begin{proof}
The comodule $N\circ(A\otimes e\otimes B)$ is $C\otimes N$ with 
coaction

$$C\otimes N \to^{\delta_3\otimes \delta} C\otimes C\otimes C\otimes A\otimes
B\otimes N \otimes D \to^{c_{4135267}} C\otimes A\otimes B\otimes C\otimes
C\otimes N\otimes D\text{.}$$

\noindent The left $C^o$ and $C$ coactions on $C\otimes N$ both are the cofree
coactions, i.e. they are determined by the comultiplications. 

The basic
property of a cofree comodule is that any comodule map $\beta : M \Rightarrow C\otimes N$ to a cofree comodule
is uniquely determined by its corestriction to $N$, by which is meant the map
$\alpha = (\epsilon\otimes N)\beta : M \Rightarrow N$ in $\mathcal{V}$. Specifically, $\beta$ can be recovered from $\alpha$ as

$$M \to^{\delta_l} C\otimes M \to^{1\otimes\alpha} C\otimes N\text{.}$$ 
 
It follows that in the setting of the lemma $\beta$ can be reconstructed from
$\alpha$ in two ways. The condition \eqref{1} asserts that these
two reconstructions are the same. 

It is easily seen that $\beta$ is a left $A$-, $B$- right $D$- comodule map
if and only if $\alpha$ is. The lemma is proved.
\end{proof}

\begin{lemma}\label{l2}
Let $\beta$ be a 2-cell:

$$
\bfig
\Atriangle/@{>}|-*@{|}`@{>}|-*@{|}`@{>}|-*@{|}/<400,400>[I`C^o\otimes C`D;n`M`N]
\morphism(400,250)/=>/<0,-100>[`;\beta]
\efig
$$

\noindent Let $\alpha : M \to N$ be a map in $\mathcal{V}$ determined by the
pasting composite:

\begin{equation}
\bfig
\dtriangle(0,0)|llb|/@{>}|-*@{|}`@{>}|-*@{|}`@{>}|-*@{|}/[I`I`C^o\otimes C;1`j`\epsilon^\ast]
\btriangle(500,0)/`@{>}|-*@{|}`@{>}|-*@{|}/[I`C^o\otimes C`D;`M`N]
\morphism(350,250)/=>/<0,-100>[`;\delta]
\morphism(650,250)/=>/<0,-100>[`;\beta]
\efig
\label{pc2}
\end{equation}

\noindent It satisfies:
 
\begin{equation}
C \to^{\alpha} N \to^{\delta_l} C\otimes C\otimes N \two^{1\otimes\epsilon\otimes1}_{\epsilon\otimes1\otimes1} C\otimes N
\label{3}
\end{equation}

\noindent The 2-cell $\beta$ is uniquely determined by a right $D$-comodule map
$\alpha$ which satisfies \eqref{3}.

\end{lemma}

\begin{proof}
 The 2-cell $\beta$ is a map $C \rightarrow
C\otimes_{C^o\otimes C}N$. This is induced by a map $\beta' : C \rightarrow C\otimes N$ satisfying

\begin{align}
(1 \otimes ((\epsilon\otimes 1)\delta_l))\beta' = (c\delta\otimes
1)\beta'
\label{d1}
\end{align}

\begin{align}
(1 \otimes ((1\otimes\epsilon)\delta_l))\beta' = (\delta\otimes 1)\beta'
\label{d2}
\end{align}

\noindent We have $\alpha = (\epsilon\otimes 1)\beta'$. From $\alpha$
we can recover $\beta'$ in two ways: using \eqref{d1} it can
be shown that $\beta'$ can be reconstructed from $\alpha$ as the top composite
in \eqref{3}, or using \eqref{d2} it can be shown that $\beta'$ can be
reconstructed from $\alpha$ as the bottom composite in
\eqref{3}. So, the map $\alpha$ defined from $\beta$
satisfies \eqref{3}. Conversely, $\beta$ can be defined from a map $\alpha$
which satisfies \eqref{3}.

It is easily checked that $\beta$ is a right $D$-comodule map if and only
if $\alpha$ is. The lemma is proved.
 \end{proof}

The maps $\xi_2 = 1\otimes \epsilon\otimes 1 : C\otimes
C\otimes C \rightarrow C\otimes C$ and $\xi_0 = \delta : C \rightarrow C\otimes C$ define a
monoidal morphism structure on $\epsilon^\ast = \epsilon^\ast\otimes\epsilon^\ast
: I \spanarr C^o\otimes C$. For any $n \geq 0$ we have a 2-cell:

\begin{equation}
\bfig
\square/@{>}|-*@{|}`@{>}|-*@{|}`@{>}|-*@{|}`@{>}|-*@{|}/<500,500>[I`(C^o\otimes
C)^{\otimes n}`I`(C^o\otimes C);\epsilon^\ast`1`p_n`\epsilon^\ast] 
\morphism(250,250)/=>/<0,-100>[`;\xi_n]
\efig
\label{xin}
\end{equation} 

\begin{lemma}\label{lmon}
For any $n$, the function defined on the set of 2-cells

$$\mathcal{C}((C^o\otimes C)^{\otimes n}, C^o\otimes C)(M, N\circ p_n)$$

\noindent with values in 

\begin{equation}
\mathcal{C}(I, C^o\otimes C)(M\circ \epsilon^\ast, N\circ \epsilon^\ast)
\label{values}
\end{equation}

\noindent taking

$$
\bfig
\Atriangle/@{>}|-*@{|}`@{>}|-*@{|}`@{>}|-*@{|}/<400,400>[(C^o\otimes C)^{\otimes
n}`C^o\otimes C`C^o\otimes C;p_n`M`A] 
\morphism(400,250)/=>/<0,-100>[`;\beta]
\efig
$$

\noindent to the pasting composite

$$
\bfig
\btriangle(500,0)/@{>}|-*@{|}`@{>}|-*@{|}`@{>}|-*@{|}/[(C^o\otimes C)^{\otimes
n}`C^o\otimes C`C^o\otimes C;p_n`M`A] 
\morphism(650,250)/=>/<0,-100>[`;\beta]
\square/@{>}|-*@{|}`@{>}|-*@{|}`@{>}|-*@{|}`@{>}|-*@{|}/<500,500>[I`(C^o\otimes
C)^{\otimes n}`I`(C^o\otimes C);\epsilon^\ast`1``\epsilon^\ast] 
\morphism(250,250)/=>/<0,-100>[`;\xi_n]
\efig
$$

\noindent is injective. 

\end{lemma}

\begin{proof}
The $n = 0$ case follows from Lemma \ref{l1}. For $n = 1$ the function
forgets the left $C^o\otimes C$ comodule structure which clearly is injective.
For $n \geq 2$, the 2-cell $\xi_n$ can be written as a pasting composite of the
2-cells $\xi_2\otimes(C^o\otimes C)^{\otimes n-1}, \ldots,
\xi_2\otimes(C^o\otimes C)^{\otimes2}, \xi_2$. Pasting from the left by each of
these is an injective function by Lemma \ref{l2}, hence pasting from the left by $\xi_n$ is injective too.
 \end{proof}  

\section{Comonads in the bicategory of comodules}\label{sec4}

We will use the lower case Greek letters $\epsilon$ and $\delta$ for counits and
comultiplications of both comonads in $\mathcal{C}$ and the comonoids. Although
these are not the same, below it will become clear that such notation is not
confusing.

Let $E$ be a comonoid. There is an equivalence of categories between comonads
on $E$ in the bicategory $\mathcal{C}$ of comodule and comonoid maps with
codomain $E$. If $\epsilon : G \rightarrow E$ is a comonoid map, then the
adjunction

$$\bfig
\morphism/@{>}|-*@{|}/[\epsilon_\ast\dashv \epsilon^\ast : E`G;]
\efig
$$ 

\noindent induces a comonad on $E$. Conversely, if $G$ is a comonad on
$E$ with comonad comultiplication $\delta : G \rightarrow G\otimes_E G$ and
comonad counit $\epsilon : G \rightarrow E$, then $G$ itself becomes a comonoid
with comultiplication and counit

$$G \to^\delta G\otimes_E G \to^i G\otimes G$$  

$$G \to^\epsilon E \to^\epsilon I$$ 

\noindent while $\epsilon : G \rightarrow E$ becomes a comonoid map. In
fact, the comonoid $G$ is the Eilenberg-Moore object of $(E, G)$ with the
universal $G$-coalgebra

$$\bfig
\Atriangle/@{>}|-*@{|}`@{>}|-*@{|}`@{>}|-*@{|}/<400,400>[G`E`E;\epsilon_\ast`\epsilon_\ast`g]
\morphism(400,250)/=>/<0,-100>[`;\delta]
\efig$$
  
\begin{proposition}
$\mathcal{C}$ admits the Eilenberg-Moore construction for comonads.
\end{proposition}

It follows from Proposition \ref{monem} that $\mathrm{Mon}\mathcal{C}$ also
admits the comonad Eilenberg-Moore construction. To wit, given a
monoidal structure on a comonad $G$, the comonoid $G$ becomes a monoidale in $\mathcal{C}$, while $\epsilon_\ast : G \spanarr
E$ becomes a strong monoidal morphism.

The correspondence between comonads and comonoid maps lifts to a
correspondence between monoidal comonads on the monoidale $E$ and monoidales $G$ in $\mathcal{C}$ together with a
comonoid map $G \rightarrow E$ such that $\epsilon_\ast : G \spanarr E$ is a
strong monoidal morphism.

\section{Quantum Categories}\label{secqc} 

Essentially following \cite{quantumcat} we define a quantum
category in $\mathcal{V}$. In
\cite{quantumcat} it was shown that a quantum category in $Set$ is the same as
a small category and a quantum category in $Vect^{\mathrm{op}}$ is the same as a
bialgebroid \cite{Takeuchi}, \cite{Lu}, \cite{Xu}. Most of the section after
the definition is dedicated to proving Statement \ref{statement}, which
translates that definition to a set of axioms close to the definitions of
bialgebroid in the literature.

\begin{definition} A quantum graph $(C, A)$ in $\mathcal{V}$
consists of a comonoid $C$ and a comonad $A$ on $C^o\otimes C$.

$$\bfig
\Vtriangle/@{>}|-*@{|}`@{>}|-*@{|}`@{>}|-*@{|}/[C^o\otimes C`C^o\otimes C`C^o\otimes C;A`A`A]
\morphism(500,350)/=>/<0, -100>[`;\delta]
\morphism(1000,500)|b|/{@{>}@/^32pt/}/<-500,-500>[C^o\otimes C`C^o\otimes
C;C\otimes C]
\place(963,87)[\dline]
\morphism(820,190)/=>/<90, -90>[`;\epsilon] 
\efig$$ 
\end{definition} 

\begin{definition}\label{def} A quantum category $(C, A)$ in
$\mathcal{V}$ consists of a comonoid $C$ together with a monoidal comonad $A$ on
$C^o\otimes C$.
\end{definition}

A quantum category has an underlying quantum graph and 2-cells 

\begin{equation}
\bfig
\square/@{>}|-*@{|}`@{>}|-*@{|}`@{>}|-*@{|}`@{>}|-*@{|}/<1000, 500>[C^o\otimes
C\otimes C^o\otimes C`C^o\otimes C\otimes C^o\otimes C`C^o\otimes C`C^o\otimes C;A\otimes A`p`p`A] 
\morphism(500,250)/=>/<0,-100>[`;\mu_2]
\Atriangle(1800,0)/@{>}|-*@{|}`@{>}|-*@{|}`@{>}|-*@{|}/<500,500>[I`C^o\otimes
C`C^o\otimes C;j`j`A] 
\morphism(2300,250)/=>/<0,-100>[`;\mu_0]
\efig
\label{mu}
\end{equation}

\noindent which make $A$ into a
monoidal morphism and both of which are comonad maps.

By Section \ref{sec4}, a quantum graph amounts to comonoids $C$, $A$ and a
comonoid map $\epsilon : A \rightarrow C^o\otimes C$. The latter itself amounts
to comonoid maps $s : A \rightarrow C^o$ and $t : A \rightarrow C$ satisfying:

$$\bfig \Ctriangle/<-``>/<250,250>[A\otimes A`A`A\otimes A;\delta``\delta]
\square(250,0)/->``>`->/<500,500>[A\otimes A`C \otimes C`A\otimes A`C\otimes
C;s\otimes t``c`t\otimes s] \efig$$

\noindent By $s$ and $t$ we can express $\epsilon$ as

$$A \to^\delta A\otimes A \to^{s\otimes t} C\otimes C\text{.}$$

\noindent $C$ is called the object of objects of the quantum graph. $A$ is
called the object of arrows. The maps $s$ and $t$ are called the source and the target maps respectively.

We regard $A$ as a comodule $C \spanarr C$ using the right $C^o\otimes
C$-coaction on it. In terms of $s$ and $t$ left and right $C$-coactions on $A$ are

$$A \to^{\delta_3} A\otimes A \otimes A \to^{1\otimes s\otimes \epsilon}
A\otimes C \to^{c^{-1}} C\otimes A\quad\text{and}$$

$$A \to^{\delta_3} A\otimes A \otimes A \to^{1\otimes \epsilon\otimes t}
A\otimes C.$$

\noindent The tensor product $H = A\otimes_C A$ of $A$ with itself over $C$ is
called the object of composable arrows for the quantum graph. It is defined by
the equalizer

$$H \to^i A\otimes A \two^{1\otimes(c^{-1}(1\otimes s\otimes
\epsilon)\delta_3)}_{(1\otimes \epsilon\otimes t)\delta_3)\otimes 1} A\otimes
C \otimes A\text{.}$$

\noindent The composite comodule

$$\bfig
\morphism/@{>}|-*@{|}/<1000,0>[C^o\otimes C\otimes C^o\otimes C`C^o\otimes C\otimes C^o\otimes C;A\otimes A]
\morphism(1000,0)/@{>}|-*@{|}/<800,0>[C^o\otimes C\otimes C^o\otimes
C`C^o\otimes C;p]
\efig
$$

\noindent is $H$ with left $C^o\otimes C\otimes C^o\otimes C$-coaction the
unique map $\delta_l : H \rightarrow C\otimes C\otimes C\otimes C\otimes H$ making

$$\bfig
\square<800,500>[H`C\otimes C\otimes C\otimes C\otimes H`A\otimes A`C\otimes
C\otimes C\otimes C\otimes A\otimes A;\delta_l`i`1\otimes1\otimes1\otimes1\otimes i`\delta_l] \efig$$

\noindent commute and right $C^o\otimes C$-coaction the unique map
$\delta_r : H \rightarrow C\otimes C\otimes H$ making

$$\bfig
\square<700,500>[H`H\otimes C\otimes C`A\otimes A`A\otimes A\otimes C\otimes
C;\delta_r`i`i\otimes1\otimes1`\delta_r] \efig$$

\noindent commute. We regard $H$ as a comodule $C \spanarr C$ using the
right $C^o\otimes C$-coaction on it.

The map $\nu_2 : H \to A$ determined by the pasting composite

\begin{equation}\label{nu2}
\bfig
\square/@{>}|-*@{|}`@{>}|-*@{|}`@{>}|-*@{|}`@{>}|-*@{|}/<600,500>[I`C^o\otimes
C\otimes C^o\otimes C`I`C^o\otimes C;\epsilon^\ast`1`p`\epsilon^\ast] 
\square(600,0)/@{>}|-*@{|}``@{>}|-*@{|}`@{>}|-*@{|}/<1000,500>[C^o\otimes
C\otimes C^o\otimes C`C^o\otimes C\otimes C^o\otimes C`C^o\otimes C`C^o\otimes C;A\otimes A``p`A] 
\morphism(1100,300)/=>/<0,-100>[`;\mu_2]
\morphism(300,300)/=>/<0,-100>[`;\xi_2]
\efig
\end{equation}

\noindent is called the composition map of the quantum category. The map $\nu_0
: C \to H$ determined by the pasting composite 

\begin{equation}\label{nu0}
\bfig
\dtriangle(0,0)|llb|/@{>}|-*@{|}`@{>}|-*@{|}`@{>}|-*@{|}/[I`I`C^o\otimes C;1`j`\epsilon^\ast]
\btriangle(500,0)/`@{>}|-*@{|}`@{>}|-*@{|}/[I`C^o\otimes C`C^o\otimes C;`j`A]
\morphism(350,250)/=>/<0,-100>[`;\xi_0]
\morphism(650,250)/=>/<0,-100>[`;\mu_0]
\efig
\end{equation}

\noindent  is called the unit map of the quantum category. 

\begin{lemma}
The 2-cells $\mu_2$ and $\mu_0$ determine a monoidal morphism structure on the
comodule $A : C^o\otimes C \spanarr C^o\otimes C$ if and only if $(A, \nu_2,
\nu_0)$ is a monoid in $\mathcal{C}(C, C)$.
\end{lemma}   

\begin{proof}
It follows from Lemma \ref{lmon} that $\mu_2$ and $\mu_0$ determine a monoidal
morphism structure on $A$ if and only if the pasting composites \eqref{nu2} and
\eqref{nu0} determine a monoidal morphism structure on $\epsilon^\ast A$.
Using the equivalence $\mathcal{C}(I, C^o\otimes C) \simeq \mathcal{C}(C, C)$,
the 2-cells \eqref{nu2} and \eqref{nu0} determine a monoidal morphism structure
on $\epsilon^\ast A$ if and only if $(A, \nu_2, \nu_0)$ is a monoid in
$\mathcal{C}(C, C)$.
 \end{proof}

Consider the pasting composite 2-cell

\begin{equation}\label{gr}
\bfig
\morphism(0,500)/@{>}|-*@{|}/<1000,0>[C^o\otimes C\otimes C^o\otimes C`C^o\otimes C\otimes C^o\otimes C;A\otimes A]
\square(1000,0)/@{>}|-*@{|}`@{>}|-*@{|}`@{>}|-*@{|}`@{>}|-*@{|}/<1000,500>[C^o\otimes C\otimes C^o\otimes C`C^o\otimes C\otimes C^o\otimes C`C^o\otimes C`C^o\otimes C;A\otimes A`p`p`A]
\morphism(0,500)/{@{>}@/^30pt/}/<2000,0>[C^o\otimes C\otimes C^o\otimes C`C^o\otimes C\otimes C^o\otimes C;A\otimes A]

\morphism(1000,700)/=>/<0,-100>[`;\delta]
\morphism(1500,300)/=>/<0,-100>[`;\mu_2]
\efig
\end{equation}

\noindent It is a map $H \rightarrow H\otimes_{(C^o\otimes C)}A$. Let $\gamma_r$
be the composite of this with the canonical injection $H\otimes_{(C^o\otimes
C)}A \rightarrow H\otimes A$. 

\begin{lemma}\label{lmud} The 2-cell $\mu_2$ is a comonad morphism if and only
if the following diagrams commute

\begin{equation}
\bfig
\square[H`A`H\otimes A`A\otimes A;\nu_2`\gamma_r`\delta`\nu_2\otimes 1]
\efig
\label{mud}
\end{equation}

\begin{equation}
\bfig
\square(1500,0)[H`A`A\otimes A`I;\nu_2`\iota`\epsilon`\epsilon\otimes\epsilon]
\efig
\label{mue}
\end{equation}

\end{lemma}

\begin{proof}
The map $\mu_2$ is a comonad map if:

$$\bfig
\square/@{>}|-*@{|}`@{>}|-*@{|}``@{>}|-*@{|}/<1000,500>[C^o\otimes C\otimes C^o\otimes C`C^o\otimes C\otimes C^o\otimes C`C^o\otimes C`C^o\otimes C;A\otimes A`p``A]
\square(1000,0)/@{>}|-*@{|}`@{>}|-*@{|}`@{>}|-*@{|}`@{>}|-*@{|}/<1000,500>[C^o\otimes C\otimes C^o\otimes C`C^o\otimes C\otimes C^o\otimes C`C^o\otimes C`C^o\otimes C;A\otimes A`p`p`A]
\morphism(0,500)/{@{>}@/^30pt/}/<2000,0>[C^o\otimes C\otimes C^o\otimes C`C^o\otimes C\otimes C^o\otimes C;A\otimes A]
\place(1000,760)[\line]
\morphism(1500,300)/=>/<0,-100>[`;\mu_2]
\morphism(500,300)/=>/<0,-100>[`;\mu_2]
\morphism(1000,700)/=>/<0,-100>[`;\delta]
\efig$$

$$=$$ 

$$\bfig
\square/`@{>}|-*@{|}``@{>}|-*@{|}/<1000,500>[C^o\otimes C\otimes C^o\otimes C``C^o\otimes C`C^o\otimes C;`p``A]
\square(1000,0)/``@{>}|-*@{|}`@{>}|-*@{|}/<1000,500>[`C^o\otimes C\otimes C^o\otimes C`C^o\otimes C`C^o\otimes C;``p`A]
\morphism(0,500)/{@{>}@/^30pt/}/<2000,0>[C^o\otimes C\otimes C^o\otimes C`C^o\otimes C\otimes C^o\otimes C;A\otimes A]
\place(1000,760)[\line]
\morphism/{@{>}@/^30pt/}/<2000,0>[C^o\otimes C`C^o\otimes C;A]
\place(1000,260)[\line]
\morphism(1000,200)/=>/<0,-100>[`;\delta]
\morphism(1000,550)/=>/<0,-100>[`;\mu_2]
\efig$$

$$
\bfig
\morphism(500,650)/=>/<0,-100>[`;\epsilon]
\morphism(0,500)/{@{>}@/^25pt/}/<1000,0>[C^o\otimes C\otimes C^o\otimes C`C^o\otimes C\otimes C^o\otimes C;A\otimes A]
\place(500,710)[\line]
\square|blrb|/@{>}|-*@{|}`@{>}|-*@{|}`@{>}|-*@{|}`@{>}|-*@{|}/<1000,500>[C^o\otimes C\otimes C^o\otimes C`C^o\otimes C\otimes C^o\otimes C`C^o\otimes C`C^o\otimes C;C^{\otimes4}`p`p`C\otimes C]
\place(530,250)[\cong]
\efig
$$

$$=$$

$$
\bfig
\morphism(500,500)/=>/<0,-100>[`;\mu_2]
\morphism(500,150)/=>/<0,-100>[`;\epsilon]
\morphism(0,500)/{@{>}@/^25pt/}/<1000,0>[C^o\otimes C\otimes C^o\otimes C`C^o\otimes C\otimes C^o\otimes C;A\otimes A]
\place(500,710)[\line]
\morphism/{@{>}@/^25pt/}/<1000,0>[C^o\otimes C`C^o\otimes C;A]
\place(500,210)[\line]
\square/`@{>}|-*@{|}`@{>}|-*@{|}`@{>}|-*@{|}/<1000,500>[C^o\otimes C\otimes C^o\otimes C`C^o\otimes C\otimes C^o\otimes C`C^o\otimes C`C^o\otimes C;
`p`p`C\otimes C]
\efig
$$

\noindent By Lemma \ref{lmon} these equalites between pasting diagrams are
equivalent to the following equalities obtained by suitably pasting to them the 2-cell
\eqref{xin} for $n = 2$.

$$\bfig
\square/@{>}|-*@{|}`@{>}|-*@{|}``@{>}|-*@{|}/<1000,500>[I`C^o\otimes C\otimes C^o\otimes C`I`C^o\otimes C;A\otimes A`I``A]
\square(1000,0)/@{>}|-*@{|}`@{>}|-*@{|}`@{>}|-*@{|}`@{>}|-*@{|}/<1000,500>[C^o\otimes C\otimes C^o\otimes C`C^o\otimes C\otimes C^o\otimes C`C^o\otimes C`C^o\otimes C;A\otimes A`p`p`A]
\morphism(0,500)/{@{>}@/^30pt/}/<2000,0>[I`C^o\otimes C\otimes C^o\otimes C;A\otimes A]
\place(1000,760)[\line]
\morphism(1500,300)/=>/<0,-100>[`;\mu_2]
\morphism(500,300)/=>/<0,-100>[`;\nu_2]
\morphism(1000,700)/=>/<0,-100>[`;\delta]
\efig$$

$$
=
$$

$$\bfig
\square/`@{>}|-*@{|}``@{>}|-*@{|}/<1000,500>[I```C^o\otimes C;`I``A]
\square(1000,0)/``@{>}|-*@{|}`@{>}|-*@{|}/<1000,500>[`C^o\otimes C\otimes C^o\otimes C`C^o\otimes C`C^o\otimes C;``p`A]
\morphism(0,500)/{@{>}@/^30pt/}/<2000,0>[I`C^o\otimes C\otimes C^o\otimes C;A\otimes A]
\place(1000,760)[\line]
\morphism/{@{>}@/^30pt/}/<2000,0>[I`C^o\otimes C;A]
\place(1000,260)[\line]
\morphism(1000,200)/=>/<0,-100>[`;\delta]
\morphism(1000,500)/=>/<0,-100>[`;\nu_2]
\efig$$

$$
\bfig
\morphism(500,300)/=>/<0,-100>[`;\xi_0]
\morphism(500,650)/=>/<0,-100>[`;\epsilon]
\morphism(0,500)/{@{>}@/^25pt/}/<1000,0>[I`C^o\otimes C\otimes C^o\otimes C;A\otimes A]
\place(500,710)[\line]
\square|blrb|/@{>}|-*@{|}`@{>}|-*@{|}`@{>}|-*@{|}`@{>}|-*@{|}/<1000,500>[I`C^o\otimes C\otimes C^o\otimes C`I`C^o\otimes C;
C^{\otimes4}`I`p`C\otimes C]
\place(1400,250)[=]
\morphism(2200,500)/=>/<0,-100>[`;\nu_2]
\morphism(2200,150)/=>/<0,-100>[`;\epsilon]
\morphism(1700,500)/{@{>}@/^25pt/}/<1000,0>[I`C^o\otimes C\otimes C^o\otimes C;A\otimes A]
\place(2200,710)[\line]
\morphism(1700,0)/{@{>}@/^25pt/}/<1000,0>[I`C^o\otimes C;A]
\place(2200,210)[\line]
\square(1700,0)/`@{>}|-*@{|}`@{>}|-*@{|}`@{>}|-*@{|}/<1000,500>[I`C^o\otimes C\otimes C^o\otimes C`I`C^o\otimes C;`I`p`A]
\efig
$$

\noindent It is easy to translate these equalities into commutative diagrams.
The first of them translates to \eqref{mud}. The second translates to the
commutativity of

\begin{equation}
\bfig
\square/>`<-``/[A`A\otimes A`H`;\delta`\nu_2``]
\square(500,0)/>``<-`/[A\otimes A`C\otimes C``A\otimes A;s\otimes t``s\otimes t`]
\morphism<1000,0>[H`A\otimes A;i]
\efig
\label{2s}
\end{equation}

\noindent This reduces to the commutativity of \eqref{2s}. Indeed, in the
diagram

$$
\bfig
\square/>`<-`<-`>/[A`A\otimes C\otimes C`H`H\otimes C\otimes C;\delta_l`\nu_2`\nu_2\otimes1\otimes1`\delta_l]
\square(500,0)/>``<-`>/<800,500>[A\otimes C\otimes C`C\otimes C`H\otimes C\otimes C`A\otimes A\otimes C\otimes C;\epsilon\otimes1\otimes1``\epsilon\otimes\epsilon\otimes1\otimes1`i\otimes1\otimes1]
\efig
$$

\noindent the left square commutes since $\nu_2$ is a right $C^o\otimes
C$-comodule map and the right square commutes given \eqref{mue}. A little calculation shows that the
outer part is exactly \eqref{2s}.
 \end{proof}

The 2-cell $\mu_0$ is a map $\mu_0 : C \rightarrow
C\otimes_{(C^o\otimes C)}A$. Let $\gamma_r$ be the
composite of this with the canonical injection $C\otimes_{(C^o\otimes C)}A
\rightarrow C\otimes A$.

\begin{lemma} The 2-cell $\mu_0$ is a comonad morphism if and only if the following diagrams commute 

\begin{equation}
\bfig
\square[C`A`C\otimes A`A\otimes A;\nu_0`\gamma_r`\delta`\nu_0\otimes 1]
\efig
\label{etad}
\end{equation}

\begin{equation}
\bfig
\Atriangle<400,400>[C`A`I;\nu_0`\epsilon`\epsilon]
\efig
\label{etae}
\end{equation} 

\end{lemma}

\begin{proof}
The map $\mu_0$ is a comonad morphism if:
$$
\bfig
\dtriangle/@{>}|-*@{|}``/[I`C^o\otimes C`;j``]
\btriangle(500,0)/`@{>}|-*@{|}`/[I``C^o\otimes C;`j`]
\Vtriangle(0,-250)/`@{>}|-*@{|}`@{<-}|-*@{|}/<500,250>[C^o\otimes C`C^o\otimes C`C^o\otimes C;`A`A]
\Atriangle(2000,0)/@{>}|-*@{|}`@{>}|-*@{|}`@{>}|-*@{|}/[I`C^o\otimes C`C^o\otimes C;j`j`]
\Vtriangle(2000,-250)/@{>}|-*@{|}`@{>}|-*@{|}`@{<-}|-*@{|}/<500,250>[C^o\otimes C`C^o\otimes C`C^o\otimes C;A`A`A]
\morphism(500,500)<0,-750>[I`C^o\otimes C;j]
\place(1500,250)[=]
\morphism(350,200)/=>/<0,-100>[`;\mu_0]
\morphism(650,200)/=>/<0,-100>[`;\mu_0]
\morphism(2500,300)/=>/<0,-100>[`;\mu_0]
\morphism(2500,-50)/=>/<0, -100>[`;\delta]
\efig
$$

$$
\bfig
\Atriangle|lra|/@{>}|-*@{|}`@{>}|-*@{|}`@{>}|-*@{|}/[I`C^o\otimes C`C^o\otimes C;j`j`A]
\morphism|b|/{@{>}@/^-25pt/}/<1000,0>[C^o\otimes C`C^o\otimes C;C\otimes C]
\place(500,-210)[\line]
\morphism(500,-50)/=>/<0,-100>[`;\epsilon]
\morphism(500,300)/=>/<0,-100>[`;\mu_0]
\place(1500,250)[=]
\morphism(2000,0)|b|/{@{>}@/^-25pt/}/<1000,0>[C^o\otimes C`C^o\otimes C;C\otimes C]
\place(2500,-210)[\line]
\Atriangle(2000,0)/@{>}|-*@{|}`@{>}|-*@{|}`/[I`C^o\otimes C`C^o\otimes C;j`j`]
\morphism(2500,100)/=>/<0,-100>[`;1_C]
\efig
$$

\noindent Using Lemma \ref{lmon} these are equivalent to the
equalities:

$$
\bfig
\dtriangle/>``/[I`I`;I``]
\btriangle(500,0)/`@{>}|-*@{|}`/[I``C^o\otimes C;`j`]
\Vtriangle(0,-250)/`@{>}|-*@{|}`@{>}|-*@{|}/<500,250>[I`C^o\otimes C`C^o\otimes C;`A`A]
\Atriangle(2000,0)[I`I`C^o\otimes C;I`j`]
\Vtriangle(2000,-250)/@{>}|-*@{|}`@{>}|-*@{|}`@{<-}|-*@{|}/<500,250>[I`C^o\otimes C`C^o\otimes C;A`A`A]
\morphism(500,500)<0,-750>[I`C^o\otimes C;j]
\place(1500,250)[=]
\morphism(350,200)/=>/<0,-100>[`;\nu_0]
\morphism(650,200)/=>/<0,-100>[`;\mu_0]
\morphism(2500,300)/=>/<0,-100>[`;\nu_0]
\morphism(2500,-50)/=>/<0, -100>[`;\delta]
\efig
$$

$$
\bfig
\Atriangle|lra|/@{>}|-*@{|}`@{>}|-*@{|}`@{>}|-*@{|}/[I`I`C^o\otimes C;I`j`A]
\morphism|b|/{@{>}@/^-25pt/}/<1000,0>[I`C^o\otimes C;C\otimes C]
\place(500,-210)[\line]
\morphism(500,-50)/=>/<0,-100>[`;\epsilon]
\morphism(500,300)/=>/<0,-100>[`;\nu_0]
\place(1500,250)[=]
\morphism(2000,0)|b|/{@{>}@/^-25pt/}/<1000,0>[I`C^o\otimes C;C\otimes C]
\place(2500,-210)[\line]
\Atriangle(2000,0)/@{>}|-*@{|}`@{>}|-*@{|}`/[I`I`C^o\otimes C;I`j`]
\morphism(2500,100)/=>/<0,-100>[`;1_C]
\efig
$$

\noindent The first of these translates to the commutativity of \eqref{etad}. The second translates to

\begin{equation}
\bfig
\square/>`>`<-`>/[A`A\otimes A`C`C\otimes C;\delta`\nu_0`s\otimes t`\delta]
\efig
\label{2s1}
\end{equation}

\noindent Which reduces to the commutativity of \eqref{2s1}. Indeed, in the
diagram

$$\bfig
\square/>`<-``>/[A`A\otimes A\otimes A`C`C\otimes C\otimes C;\delta_3`\nu_0``\delta_3]
\square(500,0)/>```>/<800,500>[A\otimes A\otimes A`A\otimes C\otimes C`C\otimes C\otimes C`C\otimes C\otimes C;s\otimes t```c\otimes 1]
\square(1300,0)/>`<-`<-`>/<800,500>[A\otimes C\otimes C`C\otimes C`C\otimes C\otimes C`C\otimes C;\epsilon\otimes1\otimes1`\nu_0\otimes1\otimes1`1`\epsilon\otimes1\otimes1]
\efig$$

\noindent the left square commutes since $\nu_0$ is a right $C^o\otimes
C$-comodule map, and the right square commutes given \eqref{etae}, while the outer part can be
seen to be \eqref{2s1}.
 \end{proof} 

Now we are in position to unpack Definition \ref{def}.

Start again with a quantum graph $(C, A)$. There is a unique map $\gamma_l : H
\to A\otimes A\otimes H$ making

\begin{equation}\label{gammal}
\bfig
\square/`>`>`>/<1000,500>[H`A\otimes A\otimes A\otimes A`A\otimes A\otimes H`A\otimes A\otimes A\otimes A;`\gamma_l`1\otimes c\otimes 1`1\otimes 1\otimes i]
\morphism(0,500)<400,0>[H`A\otimes A;i]
\morphism(400,500)<600,0>[A\otimes A`A\otimes A\otimes A\otimes A;\delta\otimes\delta]
\efig
\end{equation}

\noindent commute. By Lemma \ref{l1} the 2-cell $\mu_2$ is
determined by the map $\nu_2$. The condition of \ref{l1} says that
$\nu_2$ should respects the right coaction by $C^o\otimes C$ and the left
coactions by the first and the fourth terms in $C^o\otimes C\otimes C^o\otimes
C$ and satisfy \eqref{1}, which now becomes

\begin{equation}
H \to^{\gamma_l} A\otimes A\otimes H \two^{t\otimes\epsilon\otimes
 1}_{\epsilon\otimes s\otimes 1} C\otimes H \to^{1\otimes\nu_2} C\otimes
 A\text{.}
\label{1a}
\end{equation}

\noindent Using
\eqref{1a}, it can be shown that there exists a unque map $\gamma_r$ making

\begin{equation}\label{gammar}
\bfig
\square[H`A\otimes A\otimes H`H\otimes A`A\otimes A\otimes
A;\delta_l`\gamma_r`1\otimes 1\otimes\nu_2`i\otimes1]
\efig
\end{equation}

\noindent commute. This is the same as the map defined before by \eqref{gr}.
Observe that commutativity of the diagram \eqref{mud} in Lemma \ref{lmud}
implies that $\nu_2$ respects the left coaction by the first and the fourth terms in $C^o\otimes C\otimes C^o\otimes C$.

By Lemma \eqref{l2} the 2-cell $\mu_0$ is determined by the map $\nu_0$. The
condition of \eqref{l2} says that $\nu_0$ should respect the right $C^o\otimes
C$ coaction and satisfy \eqref{3}, which now becomes

\begin{equation}\label{ngr}
C \to^{\nu_0} A \two^{(s\otimes1)\delta}_{(t\otimes1)\delta}
C\otimes A\text{.}
\end{equation}

\noindent The common value of the two composites in \eqref{ngr} is the same as
the map $\gamma_r$ defined above. Observe that if $(A, \nu_2, \nu_0)$ is a
comonoid in $\mathcal{C}(C, C)$, then \eqref{ngr} follows from \eqref{1a}.

Assembling the established facts we obtain: 
 
\begin{statement}[B. Day, R. Street]\label{statement} Giving a quantum category
structure on a quantum graph $(C, A)$ is equivalent to giving maps $\nu_2 : H \rightarrow A$ and $\nu_0
: C \rightarrow A$ satisfying the following axioms:

Axiom 1: $(A, \nu_2, \nu_0)$ is a monoid in $\mathcal{C}(C, C)$.

Axiom 2: The diagram \eqref{1a}, in which the map $\gamma_l$ is defined by
\eqref{gammal}.

Axiom 3: The diagram \eqref{mud}, in which the map $\gamma_r$ is defined by
\eqref{gammar} using Axiom 2.

Axiom 4: The diagram \eqref{mue}.

Axiom 5:  The diagram \eqref{etad}, in which the map $\gamma_r$ is defined by
\eqref{ngr}.

Axiom 6: The diagram \eqref{etae}.
\endstatement
\end{statement}

For more clarity see the Appendix where these axioms are presented using
string diagrams.

By Section \ref{sec4}, a quantum category
structure on a quantum graph $(C, A)$ is the same as a monoidale structure on
$A$, such that $\epsilon : A \rightarrow C^o\otimes C$ is strong monoidal. 
In term of our data this monoidal structure on $A$ can be expressed as follows.
The multiplication $A\otimes A \spanarr A$ is $H$ with left
and right coactions the maps $\gamma_l : H \rightarrow A\otimes A\otimes H$ and
$\gamma_r : H \rightarrow H\otimes A$. The unit $I \spanarr A$ is $C$ with the
right coaction the map $\gamma_r : C \rightarrow C\otimes A$. The monoidale $A$
is the Eilenberg-Moore object of the comonad $A: C^o\otimes C \spanarr
C^o\otimes C$ in $\mathrm{MonComod}(\mathcal{V})$. Applying the
representable pseudofunctor $\mathrm{Mon}\mathcal{C}(I, -) :
\mathrm{Mon}\mathcal{C} \rightarrow \mathrm{Mon}Cat$ to the universal 2-cell
  
$$
\bfig
\Atriangle/@{>}|-*@{|}`@{>}|-*@{|}`@{>}|-*@{|}/<400,400>[A`C^o\otimes
C`C^o\otimes C;\epsilon_\ast`\epsilon_\ast`A] \morphism(400,250)/=>/<0,-100>[`;]
\efig
$$ 

\noindent we obtain an Eilenberg-Moore construction in the category of monoidal
categories and monoidal functors:

$$
\bfig
\Atriangle<400,400>[\mathrm{Comod}(I, A)`\mathcal{C}(I, C^o\otimes
C)`\mathcal{C}(I, C^o\otimes C);``A\circ -] 
\morphism(400,250)/=>/<0,-100>[`;]
\efig
$$ 

\noindent Using the equivalence \eqref{CC} we can transport the monoidal
comonad $A\circ -$ on the category $\mathcal{C}(I, C^o\otimes C)$ to a monoidal comonad
on the category $\mathcal{C}(C, C)$. Thus, a quantum category defines a monoidal
comonad on $\mathcal{C}(C, C)$, the Eilenberg-Moore object of which is the
category of the right $A$-comodules.
     
\section{The category of quantum categories}

Suppose that $f : C \to C'$ is a comonoid map. Let
$\alpha : f^\ast\circ f_\ast \to C$ and $\beta : C' \to f_\ast\circ f^\ast$ be
the unit and the counit of the adjunction $f^{\ast} \dashv f_\ast$ as in Section \ref{bc}. They are maps respectively in
$\mathcal{C}(C, C)$ and $\mathcal{C}(C', C')$. By biduality, from $\alpha$ we
get a map $\alpha : e \longrightarrow e\circ(f_\ast\otimes f^{\ast o})$ in
$\mathcal{C}(C\otimes C^o, I)$ and from $\beta$ we get a map $\beta :
(f_\ast\otimes f^{\ast o})\circ n \Rightarrow n$ in $\mathcal{C}(I, C'^o\otimes
C')$.

The comodule $f^{\ast o}\otimes f_\ast : C^o\otimes C
\spanarr C'^o\otimes C'$ is an opmonoidal morphism with structure
2-cells

$$\bfig
\square/@{>}|-*@{|}`@{<-}|-*@{|}`@{<-}|-*@{|}`@{>}|-*@{|}/<1300,500>[C\otimes
C^o`C'^o\otimes C'`C^o\otimes C\otimes C^o\otimes C`C'^o\otimes C'\otimes
C'^o\otimes C';f^{\ast o}\otimes f_\ast`p`p`f^{\ast o}\otimes f_\ast\otimes
f^{\ast o}\otimes f_\ast] \morphism(650,300)/=>/<0,-100>[`;\omega_2] 
\Vtriangle(2000,0)/@{>}|-*@{|}`@{<-}|-*@{|}`@{<-}|-*@{|}/[C^o\otimes
C`C'^o\otimes C'`I;f^{\ast o}\otimes f_\ast`j`j]
\morphism(2500,300)/=>/<0,-100>[`;\omega_0]
\efig
$$ 
 
\noindent defined to be $\omega_2 = f^{\ast o}\otimes\alpha\otimes f_\ast$ and
$\omega_0 = \beta$.

Suppose that $g : C \rightarrow C'$ is another comonoid map. The opmonoidal map
$f^{\ast o}\otimes f_\ast$ acts from the left on $f^{\ast o}\otimes g_\ast$ by

$$\bfig
\square/@{>}|-*@{|}`@{<-}|-*@{|}`@{<-}|-*@{|}`@{>}|-*@{|}/<1300,500>[C\otimes C^o`C'^o\otimes C'`C^o\otimes C\otimes C^o\otimes C`C'^o\otimes C'\otimes C'^o\otimes C';f^{\ast o}\otimes g_\ast`p`p`f^{\ast o}\otimes f_\ast\otimes f^{\ast o}\otimes g_\ast]
\morphism(650,300)/=>/<0,-100>[`;\lambda_l]
\efig
$$

\noindent defined as $\lambda_l = f^{\ast o}\otimes\alpha\otimes g_\ast$.
Similarly $g^{\ast o}\otimes g_\ast$ acts from the right on $f^{\ast o}\otimes
g_\ast$ with coaction 2-cell $\lambda_r$ defined similarlly.

\begin{definition} A map between quantum graphs $(\sigma, f) : (C, A)
\rightarrow (C', A')$ consists of a morphism of comonoids $f : C \rightarrow C'$ and a 2-cell

\begin{equation}
\bfig
\square/@{>}|-*@{|}`@{>}|-*@{|}`@{>}|-*@{|}`@{>}|-*@{|}/<600, 500>[C^o\otimes
C`C^o\otimes C`C'^o\otimes C'`C'^o\otimes C';A`f^{\ast o}\otimes f_\ast`f^{\ast o}\otimes f_\ast`A'] 
\morphism(300,300)/=>/<0, -100>[`;\sigma]
\efig
\label{square}
\end{equation}

\noindent such that $(f^{\ast o}\otimes f_\ast, \sigma)$ is a comonad map. 
\end{definition}

\begin{definition} A functor $(f, \sigma) : (C, A) \rightarrow (C', A')$ between
quantum categories is a map between the underlying quantum graphs such that the
2-cell $\sigma$ is a square. 
\end{definition}

In other words a quantum functor is an opmorphism of monoidal
comonads of the form $(f^{\ast o}\otimes f_\ast, \varphi) : (C, A) \to (C',
A')$. Here are the equalities that $\sigma$ must satisfy:
 
$$
\bfig
\Atriangle(0,400)|mab|/>``@{>}|-*@{|}/<500,400>[C^o\otimes C\otimes C^o\otimes C`C'^o\otimes C'\otimes C'^o\otimes C'`C'^o\otimes C'\otimes C'^o\otimes C';f^{\ast o}\otimes f_\ast\otimes f^{\ast o}\otimes f_\ast``A'\otimes A']
\Vtriangle(500,400)|aam|/@{>}|-*@{|}``>/<500,400>[C^o\otimes C\otimes C^o\otimes C`C^o\otimes C\otimes C^o\otimes C`C'^o\otimes C'\otimes C'^o\otimes C';A\otimes A``f^{\ast o}\otimes f_\ast\otimes f^{\ast o}\otimes f_\ast]
\Atriangle(1000,400)/`@{>}|-*@{|}`/<500,400>[C^o\otimes C\otimes C^o\otimes C`C'^o\otimes C'\otimes C'^o\otimes C'`C^o\otimes C;`p`]
\Vtriangle(0,0)/`@{>}|-*@{|}`/<500,400>[C'^o\otimes C'\otimes C'^o\otimes C'`C'^o\otimes C'\otimes C'^o\otimes C'`C'^o\otimes C';`p`]
\Atriangle(500,0)/`@{>}|-*@{|}`@{>}|-*@{|}/<500,400>[C'^o\otimes C'\otimes C'^o\otimes C'`C'^o\otimes C'`C'^o\otimes C';`p`A']
\Vtriangle(1000,0)/``@{>}|-*@{|}/<500,400>[C'^o\otimes C'\otimes C'^o\otimes C'`C^o\otimes C`C'^o\otimes C';``f^{\ast o}\otimes f_\ast]
\morphism(850,250)|b|/=>/<0, -100>[`;\mu]
\morphism(850,650)|b|/=>/<0, -100>[`;\sigma\otimes \sigma]
\morphism(1550,450)|b|/=>/<-100,-100>[`;\omega_2]
\efig
$$

$$=$$

$$
\bfig
\Atriangle(0,400)|maa|/>`@{>}|-*@{|}`/<500,400>[C^o\otimes C\otimes C^o\otimes
C`C'^o\otimes C'\otimes C'^o\otimes C'`C^o\otimes C;f^{\ast o}\otimes f_\ast\otimes f^{\ast o}\otimes f_\ast`p`] \Vtriangle(500,400)/@{>}|-*@{|}``/<500,400>[C^o\otimes C\otimes C^o\otimes C`C^o\otimes C\otimes C^o\otimes C`C^o\otimes C;A\otimes A``]
\Atriangle(1000,400)/`@{>}|-*@{|}`@{>}|-*@{|}/<500,400>[C^o\otimes C\otimes C^o\otimes C`C^o\otimes C`C^o\otimes C;`p`A]
\Vtriangle(0,0)/`@{>}|-*@{|}`@{>}|-*@{|}/<500,400>[C'^o\otimes C'\otimes C'^o\otimes C'`C^o\otimes C`C'^o\otimes C';`p`f^{\ast o}\otimes f_\ast]
\Atriangle(500,0)/``@{>}|-*@{|}/<500,400>[C^o\otimes C`C'^o\otimes C'`C'^o\otimes C';``A']
\Vtriangle(1000,0)/``@{>}|-*@{|}/<500,400>[C^o\otimes C`C^o\otimes C`C'^o\otimes C';``f^{\ast o}\otimes f_\ast]
\morphism(1300,250)|b|/=>/<0, -100>[`;\sigma]
\morphism(1300,650)|b|/=>/<0, -100>[`;\mu]
\morphism(550,450)|b|/=>/<-100,-100>[`;\omega_2]
\efig
$$

$$
\bfig
\square(400,0)|ammb|/@{>}|-*@{|}`>`>`@{>}|-*@{|}/<600, 500>[C^o\otimes C`C^o\otimes C`C'^o\otimes C'`C'^o\otimes C';A`f^{\ast o}\otimes f_\ast`f^{\ast o}\otimes f_\ast`A']
\morphism(700,325)/=>/<0, -100>[`;\sigma]
\morphism(0,900)/@{>}|-*@{|}/<300,-300>[I`;j]
\morphism(0,900)|l|/@{>}|-*@{|}/<300,-800>[I`;j]
\morphism(0,900)/@{>}|-*@{|}/<900,-300>[I`;j]
\morphism(400,700)/=>/<0, -100>[`;\eta]
\morphism(370,420)/=>/<-100, -100>[`;\omega_0]
\place(1350,400)[=]
\square(2000,0)|ammb|/``>`@{>}|-*@{|}/<600, 500>[`C^o\otimes C`C'^o\otimes C'`C'^o\otimes C';``f^{\ast o}\otimes f_\ast`A']
\morphism(1500,900)/@{>}|-*@{|}/<900,-800>[I`;j]
\morphism(1500,900)|l|/@{>}|-*@{|}/<400,-800>[I`;j]
\morphism(1500,900)/@{>}|-*@{|}/<900,-300>[I`;j]
\morphism(2000,325)/=>/<0, -120>[`;\eta]
\morphism(2300,500)/=>/<0, -120>[`;\omega_0]  
\efig
$$

A map of quantum graphs $(C, A) \to (C', A)$ amounts to comonoid maps $f :
C \to C'$ and $\varphi : A \rightarrow A'$ for which the diagrams

\begin{equation} 
\bfig
\square(250,800)[A`C`A'`C';s`\varphi`f`s]
\square(1600,800)[A`C`A'`C';t`\varphi`f`t]
\efig
\label{mg}
\end{equation}

\noindent commute. The pair $(f, \varphi)$ is a functor of quantum categories
if additionally it satisfies:

$$
\bfig
\Ctriangle/>``>/<250,250>[A\otimes_C A`A\otimes_{C'} A`A'\otimes_{C'}
A';\iota``\varphi\otimes_{C'} \varphi] 
\square(250,0)/>``>`>/<500,500>[A\otimes_C A`A`A'\otimes_{C'}
A'`A';\nu_2``\varphi`\nu_2] \square(1600,0)[C`A`C'`A';\nu_0`f`\varphi`\nu_0]
\efig
$$

\noindent The tensor product $A\otimes_{C'} A$ of $A$ with itself over
$C'$ is taken by regarding $A$ as a comodule $C' \spanarr C'$ with
left and right coactions:

$$A \to^{\delta} A\otimes A \to^{1\otimes f} A\otimes A' \to^{1\otimes s}
A\otimes C' \to^c C'\otimes A$$

$$A \to^{\delta} A\otimes A \to^{1\otimes f} A\otimes A' \to^{1\otimes s}
A\otimes C'.$$

\noindent  Observe that in a quantum functor the map $f$ is determined by
the map $\varphi$ via $f = \varphi\nu_0$.

The notion of the quantum functor includes the notion of functor between
small categories and the notion of weak morphism of bialgebroids \cite{S}.

By Section 2 an opmorphism between monoidal comonads is determined by
an opmonoidal morphism between the Eilenberg-Moore objects. Thus given a
quantum functor $(f, \varphi) : (C, A) \rightarrow (C', A')$ the comodule
$\varphi_\ast : A \spanarr A'$ has an opmonoidal morphism structure which lifts
the opmonoidal morphism structure on $f^{\ast o}\otimes f_\ast^o : C^o\otimes C
\spanarr C^o\otimes C$. By application of the functor $\mathrm{MonComod}(I, -)
: \mathrm{MonComod} \rightarrow \mathrm{Mon}Cat$ we get an opmonoidal functor
between categories of right $A$-comodules:

$$\mathrm{Comod}(I, A) \to^{\mathrm{MonComod}(I, \varphi_\ast)}
\mathrm{Comod}(I, A')\text{.}$$

Define composition of quantum functors $\varphi : (C, A) \to
(C', A')$ and $\varphi' : (C', A') \to (C'', A'').$ by 

$$A \to^{\varphi} A' \to^{\varphi'} A''.$$

\noindent The units for this composition are provided by quantum functors of
the form $(1, 1)$.

\begin{theorem} Quantum categories and functors between them form a category
$\mathrm{qCat}\mathcal{V}$.
\end{theorem}

\begin{definition} A natural transformation $\tau
: (f, \varphi) \Rightarrow (g, \varphi')$

$$\bfig
\morphism(0,0)|a|/{@{>}@/^1.3em/}/<700,0>[(A,C)`(A',C');(f, \varphi)]
\morphism(0,0)|b|/{@{>}@/_1.3em/}/<700,0>[(A,C)`(A',C');(g, \varphi')]
\morphism(350,50)|a|/=>/<0,-100>[`;\tau]
\efig$$

\noindent between functors is a 2-cell

$$\bfig
\square/@{>}|-*@{|}`@{>}|-*@{|}`@{>}|-*@{|}`@{>}|-*@{|}/<600, 500>[C^o\otimes
C`C^o\otimes C`C'^o\otimes C'`C'^o\otimes C';A`f^{\ast o}\otimes g_\ast`f^{\ast o}\otimes g_\ast`A'] 
\morphism(300,300)/=>/<0,-100>[`;\tau]
\efig
$$

\noindent making $f^{\ast o}\otimes g_\ast$ into a comonad map so that both the
left coaction of $f^{\ast o}\otimes f_\ast$ and the right coaction of $g^{\ast
o}\otimes g_\ast$ on $f^{\ast o}\otimes g_\ast$ respect the comonad structure. 
\end{definition}

Here are the equalities which the 2-cell $\tau$ should satisfy.

$$
\bfig
\Atriangle(0,400)|mab|/>``@{>}|-*@{|}/<500,400>[C^o\otimes C\otimes C^o\otimes C`C'^o\otimes C'\otimes C'^o\otimes C'`C'^o\otimes C'\otimes C'^o\otimes C';f^{\ast o}\otimes f_\ast\otimes f^{o\ast}\otimes g_\ast``A'\otimes A']
\Vtriangle(500,400)|aam|/@{>}|-*@{|}``>/<500,400>[C^o\otimes C\otimes C^o\otimes C`C^o\otimes C\otimes C^o\otimes C`C'^o\otimes C'\otimes C'^o\otimes C';A\otimes A``f^{\ast o}\otimes f_\ast\otimes f^{\ast o}\otimes g_\ast]
\Atriangle(1000,400)/`@{>}|-*@{|}`/<500,400>[C^o\otimes C\otimes C^o\otimes C`C'^o\otimes C'\otimes C'^o\otimes C'`C^o\otimes C;`p`]
\Vtriangle(0,0)/`@{>}|-*@{|}`/<500,400>[C'^o\otimes C'\otimes C'^o\otimes C'`C'^o\otimes C'\otimes C'^o\otimes C'`C'^o\otimes C';`p`]
\Atriangle(500,0)/`@{>}|-*@{|}`@{>}|-*@{|}/<500,400>[C'^o\otimes C'\otimes C'^o\otimes C'`C'^o\otimes C'`C'^o\otimes C';`p`A']
\Vtriangle(1000,0)/``@{>}|-*@{|}/<500,400>[C'^o\otimes C'\otimes C'^o\otimes C'`C^o\otimes C`C'^o\otimes C';``f^{\ast o}\otimes g_\ast]
\morphism(850,250)|b|/=>/<0, -100>[`;\mu_2]
\morphism(850,650)|b|/=>/<0, -100>[`;\varphi\otimes \tau]
\morphism(1550,450)|b|/=>/<-100,-100>[`;\lambda_l]
\efig
$$

$$=$$

$$
\bfig
\Atriangle(0,400)|maa|/>`@{>}|-*@{|}`/<500,400>[C^o\otimes C\otimes C^o\otimes C`C'^o\otimes C'\otimes C'^o\otimes C'`C^o\otimes C;f^{\ast o}\otimes f_\ast\otimes f^{o\ast}\otimes g_\ast`p`]
\Vtriangle(500,400)/@{>}|-*@{|}``/<500,400>[C^o\otimes C\otimes C^o\otimes C`C^o\otimes C\otimes C^o\otimes C`C^o\otimes C;A\otimes A``]
\Atriangle(1000,400)/`@{>}|-*@{|}`@{>}|-*@{|}/<500,400>[C^o\otimes C\otimes C^o\otimes C`C^o\otimes C`C^o\otimes C;`p`A]
\Vtriangle(0,0)/`@{>}|-*@{|}`@{>}|-*@{|}/<500,400>[C'^o\otimes C'\otimes C'^o\otimes C'`C^o\otimes C`C'^o\otimes C';`p`f^{\ast o}\otimes g_\ast]
\Atriangle(500,0)/``@{>}|-*@{|}/<500,400>[C^o\otimes C`C'^o\otimes C'`C'^o\otimes C';``A']
\Vtriangle(1000,0)/``@{>}|-*@{|}/<500,400>[C^o\otimes C`C^o\otimes C`C'^o\otimes C';``f^{\ast o}\otimes g_\ast]
\morphism(1300,250)|b|/=>/<0, -100>[`;\tau]
\morphism(1300,650)|b|/=>/<0, -100>[`;\mu_2]
\morphism(550,450)|b|/=>/<-100,-100>[`;\lambda_l]
\efig
$$

$$
\bfig
\Atriangle(0,400)|mab|/>``@{>}|-*@{|}/<500,400>[C^o\otimes C\otimes C^o\otimes C`C'^o\otimes C'\otimes C'^o\otimes C'`C'^o\otimes C'\otimes C'^o\otimes C';f^{\ast o}\otimes g_\ast\otimes g^{\ast o}\otimes g_\ast``A'\otimes A']
\Vtriangle(500,400)|aam|/@{>}|-*@{|}``>/<500,400>[C^o\otimes C\otimes C^o\otimes C`C^o\otimes C\otimes C^o\otimes C`C'^o\otimes C'\otimes C'^o\otimes C';A\otimes A``f^{\ast o}\otimes g_\ast\otimes g^{\ast o}\otimes g_\ast]
\Atriangle(1000,400)/`@{>}|-*@{|}`/<500,400>[C^o\otimes C\otimes C^o\otimes C`C'^o\otimes C'\otimes C'^o\otimes C'`C^o\otimes C;`p`]
\Vtriangle(0,0)/`@{>}|-*@{|}`/<500,400>[C'^o\otimes C'\otimes C'^o\otimes C'`C'^o\otimes C'\otimes C'^o\otimes C'`C'^o\otimes C';`p`]
\Atriangle(500,0)/`@{>}|-*@{|}`@{>}|-*@{|}/<500,400>[C'^o\otimes C'\otimes C'^o\otimes C'`C'^o\otimes C'`C'^o\otimes C';`p`A']
\Vtriangle(1000,0)/``@{>}|-*@{|}/<500,400>[C'^o\otimes C'\otimes C'^o\otimes C'`C^o\otimes C`C'^o\otimes C';``f^{\ast o}\otimes g_\ast]
\morphism(850,250)|b|/=>/<0, -100>[`;\mu_2]
\morphism(850,650)|b|/=>/<0, -100>[`;\tau\otimes\varphi']
\morphism(1550,450)|b|/=>/<-100,-100>[`;\lambda_r]
\efig
$$

$$=$$

$$
\bfig
\Atriangle(0,400)|maa|/>`@{>}|-*@{|}`/<500,400>[C^o\otimes C\otimes C^o\otimes C`C'^o\otimes C'\otimes C'^o\otimes C'`C^o\otimes C;f^{\ast o}\otimes g_\ast\otimes g^{\ast o}\otimes g_\ast`p`]
\Vtriangle(500,400)/@{>}|-*@{|}``/<500,400>[C^o\otimes C\otimes C^o\otimes C`C^o\otimes C\otimes C^o\otimes C`C^o\otimes C;A\otimes A``]
\Atriangle(1000,400)/`@{>}|-*@{|}`@{>}|-*@{|}/<500,400>[C^o\otimes C\otimes C^o\otimes C`C^o\otimes C`C^o\otimes C;`p`A]
\Vtriangle(0,0)/`@{>}|-*@{|}`@{>}|-*@{|}/<500,400>[C'^o\otimes C'\otimes C'^o\otimes C'`C^o\otimes C`C'^o\otimes C';`p`f^{\ast 0}\otimes g_\ast]
\Atriangle(500,0)/``@{>}|-*@{|}/<500,400>[C^o\otimes C`C'^o\otimes C'`C'^o\otimes C';``A']
\Vtriangle(1000,0)/``@{>}|-*@{|}/<500,400>[C^o\otimes C`C^o\otimes C`C'^o\otimes C';``f^{\ast o}\otimes g_\ast]
\morphism(1300,250)|b|/=>/<0, -100>[`;\tau]
\morphism(1300,650)|b|/=>/<0, -100>[`;\mu_2]
\morphism(550,450)|b|/=>/<-100,-100>[`;\lambda_r]
\efig
$$    

A natural transformation $\tau : (f, \varphi) \to (g, \varphi') : (C, A) \to
(C', A')$ amounts to a comonoid map $\tau : A \to A'$ such that:

\begin{equation}
\bfig
\square(250,800)[A`C`A'`C';s`\tau`f`s]
\square(1500,800)[A`C`A'`C';t`\tau`g`t]
\efig
\label{mg}
\end{equation}
 
$$\bfig
\dtriangle/<-``/<300,300>[A'\otimes_C A'`A\otimes_C
A`;\varphi\otimes_C\tau``] 
\square(300,0)/>```/<600,300>[A'\otimes_CA'`A'\otimes_{C'}
A'``;i```] 
\qtriangle(0,-300)/`>`/<300,300>[A\otimes_C A``A'\otimes_C
A';`\tau\otimes_C\varphi`] 
\square(300,-300)/```>/<600,300>[``A'\otimes_C A'`A'\otimes_{C'}
A';```i]
\Dtriangle(900,-300)/`>`<-/<300,300>[A'\otimes_{C'}
A'`A'`A'\otimes_{C'}A';`\nu_2`\nu_2]
\morphism<600,0>[A\otimes_C A`A;\nu_2]
\morphism(600,0)<550,0>[A`A';\tau]
\efig
$$

\section{$\mathrm{qCat}$ as a functor}

A coreflexive-equalizer-preserving braided strong-monoidal functor $\mathcal{V}
\rightarrow \mathcal{W}$ defines a functor between the categories of quantum
categories $\mathrm{qCat}\mathcal{V} \rightarrow \mathrm{qCat}\mathcal{W}$. Thus
$\mathrm{qCat}$ can be viewed as a 2-functor from the 2-category of braided
monoidal categories (satisfying the condition at the beginning of Section 3) and braided
strong monoidal functors to the 2-category of categories. This functor preserves
finite products since we have isomorphisms

$$\mathrm{qCat}(\mathcal{V}\times\mathcal{W}) \cong
\mathrm{qCat}\mathcal{V}\times\mathrm{qCat}\mathcal{W}$$

$$\mathrm{qCat}(\mathrm{1}) \cong \mathrm{1}$$

\begin{example} When $\mathcal{V}$ is a symmetric monoidal category, then the
functors $- \otimes - : \mathcal{V}\times\mathcal{V} \rightarrow \mathcal{V}$,
$I : 1 \rightarrow \mathcal{V}$ are symmetric monoidal. From them we obtain functors

$$- \otimes - : \mathrm{qCat}\mathcal{V}\times\mathrm{qCat}\mathcal{V}
\rightarrow \mathrm{qCat}\mathcal{V}$$

$$I : \mathrm{1} \rightarrow \mathrm{qCat}\mathcal{V}$$

\noindent defining a monoidal structure on $\mathrm{qCat}\mathcal{V}$.
\end{example}

\begin{example}\label{e1}  
There is a functor $Set \rightarrow \mathcal{V}$, taking a set $S$ to the
$S$-fold coproduct $S\cdot I$ of the monoidal unit, provided these
copowers exist. When a certain distributivity law is statisfied, this functor is
strong monoidal. Any coreflexive equalizer in $Set$ (which does not involve an
empty set) is split, and thus preserved by any functor. We have
$\mathrm{qCat}Set = Cat$ \cite{quantumcat}. So we get a functor:

$$Cat \rightarrow \mathrm{qCat}\mathcal{V}$$
 
\noindent To wit any category determines a quantum category in any (sufficiently
good) monoidal category.
\end{example}

\begin{example}\label{e2}
Suppose that $\mathcal{V}$ has small coproducts, and assume that each of
the functors $X + -$ preserves coreflexive equalizers. For any finite set $S$
let $S\cdot V$ stand for the $S$-fold coproduct of an object $V$ of $\mathcal{V}$. There is a
coreflexive-equalizer-preserving braided strong-monoidal functor $-\cdot - :
Set_f\times\mathcal{V} \rightarrow \mathcal{V}$. The preservation of coreflexive
equalizeres is due to Lemma 0.17 in \cite{J}.  We have $\mathrm{qCat}Set_f =
Cat_f$. Thus we obtain a functor
 
\begin{equation}\label{cotensor}
Cat_f\times\mathrm{qCat}\mathcal{V} \rightarrow
\mathrm{qCat}\mathcal{V}
\end{equation}

\noindent Let $i: \Delta \rightarrow Cat_f$  be the canonical embedding of the
simplicial category into the category of finite categories. Precomposing
\eqref{cotensor} with $i\times1$ we obtain a functor

$$\Delta\times\mathrm{qCat}\mathcal{V} \rightarrow \mathrm{qCat}\mathcal{V}$$

Let $A$ and $B$ be quantum categories in $\mathcal{V}$. Consider the simplicial
set

\begin{equation}\label{hss}
\mathrm{qCat}\mathcal{V}(-\cdot A, B) \text{.}
\end{equation} 

\noindent The 1-simplexes of this simplicial set are the quantum functors
from $A$ to $B$. The 2-simplexes are the quantum natural transformation.
For quantum categories $A$, $B$ and $C$, there is a map of simplicial sets

\begin{equation}\label{composition}
\mathrm{qCat}\mathcal{V}(-\cdot A, B)\times
\mathrm{qCat}\mathcal{V}(-\cdot B, C) \rightarrow \mathrm{qCat}\mathcal{V}(-\cdot A, C)
\end{equation}

\noindent defined on the $n$-simplexes by

$$\mathrm{qCat}\mathcal{V}(n\cdot A, B)\times \mathrm{qCat}\mathcal{V}(n\cdot B,
C) \stackrel{(n\cdot-)\times1}{\longrightarrow}
\mathrm{qCat}\mathcal{V}(n\cdot n\cdot A, n\cdot B)\times \mathrm{qCat}\mathcal{V}(n\cdot B, C)
$$
$$\stackrel{\mathrm{qCat}\mathcal{V}(\delta_{(n, A)},
1)\times1}{\longrightarrow} \mathrm{qCat}\mathcal{V}(n\cdot A, n\cdot B)\times
\mathrm{qCat}\mathcal{V}(n\cdot B, C) \stackrel{\mathrm{comp}}{\longrightarrow}
\mathrm{qCat}\mathcal{V}(n\cdot A, C)\text{.}$$

\noindent Here $\delta$ is the natural transformation

$$
\bfig
\Vtriangle[Cat\times
\mathrm{qCat}\mathcal{V}`Cat\times
Cat\times\mathrm{qCat}\mathcal{V}`\mathcal{V};\mathrm{diag}\times1`-\cdot-`-\cdot-\cdot-]
\morphism(500,350)/=>/<0,-100>[`;\delta]
\efig
$$

\noindent obtained by applying $\mathrm{qCat}-$ to the obvious natural
transformation

$$
\bfig
\Vtriangle[Set\times
\mathcal{V}`Set\times Set\times \mathcal{V}`
\mathcal{V};\mathrm{diag}\times1`-\cdot-`-\cdot-\cdot-]
\morphism(500,350)/=>/<0,-100>[`;\delta]
\efig
$$

\noindent Using \eqref{composition} as compositions we can construct a
simplicial set enriched category with objects the quantum categories in
$\mathcal{V}$ and a hom simplicial set from $A$ to $B$ given by \eqref{hss}.

\begin{theorem}
For any (sufficiently good) braided monoidal category $\mathcal{V}$ the above
defines a simplicial set enriched category $\mathrm{qCat}\mathcal{V}$.
\end{theorem}
\end{example}

\begin{example}
Consider the category of families $\mathrm{Fam}\mathcal{V}$. An object of
$\mathrm{Fam}\mathcal{V}$ is a pair $(S, \{A_s\})$, where $S$ is a set and
$\{A_s\}$ is an $S$ indexed family of objects of $\mathcal{V}$. A morphism $(f,
\{\varphi_s\}) : (S, \{A_s\}) \to (S', (A_s'))$ consists of a map $f : S \to S'$
and for each $s$ in $S$ a morphism $ \varphi_s : A_s \to A_{f(s)}$ in
$\mathcal{V}$. The monoidal structure on $\mathcal{V}$ induces a monoidal
structure on $\mathrm{Fam}\mathcal{V}$ in the obvious way. Consider the functor

$$\mathrm{Fam}\mathcal{V} \to \mathcal{V}\text{.}$$

\noindent taking $(S, \{A_s\})$ to the coproduct $\coprod A_s$. This functor
is monoidal and under mild conditions on $\mathcal{V}$, it preserves
coreflexive equalizers. By applying the $\mathrm{qCat}$ we obtain a functor

$$\mathrm{qCatFam}\mathcal{V} \to \mathrm{qCat}\mathcal{V}\text{.}$$

The functor 

$$\mathrm{Fam}\mathcal{V} \to Set$$

\noindent taking $(S, \{A_s\})$ to the set $S$ determines a functor

\begin{equation}\label{fam}\mathrm{qCatFam}\mathcal{V} \to
Cat\text{.}\end{equation}

\noindent In this way each object in $\mathrm{qCatFam}\mathcal{V}$ has an
underlying category.

Next we show how Hopf group coalgebras introduced in
\cite{T} are quantum categories (groupoids \cite{quantumcat}).   

Let $G$ denote a group. A Hopf $G$-coalgebra consists of a family of
algebras $\{A_g\}$ indexed by elements $g$ of $G$ together with a family of
linear maps $\{A_{gg'} \to A_g\otimes A_{g'}\}$ and an antipode satisfying certain axioms. 
Such a Hopf group coalgebra without an
antipode is an object of $\mathrm{qCatFam}Vect^{\mathrm{op}}$ with the object of
objects $(1, \{I\})$ and the object of arrows $(G, \{A_g\})$. The underlying
category is the group $G$. Using the functor \eqref{fam} from a Hopf group
coalgebra we obtain a quantum category in $Vect^{\mathrm{op}}$, in other words a bialgebroid. 
The antipode would make this bialgebroid into a Hopf algebroid.

\end{example}

\section{Appendix: Computations for quantum categories in string diagrams}
We introduce framed string diagrams to represent morphisms in a braided
monoidal bicategory as an enrichment of the string diagrams of
\cite{JS}. These diagrams are designed to ease presentation of quantum
structures. 

A string diagram of \cite{JS} has edges labeled by objects of $\mathcal{V}$ and
nodes labeled by morphisms of $\mathcal{V}$ and represents
a morphism in $\mathcal{V}$. For example a morphism $f : X\otimes Y \to Z$ in
$\mathcal{V}$ is represented by a string diagram:

\begin{center}  
\includegraphics[scale=0.4]{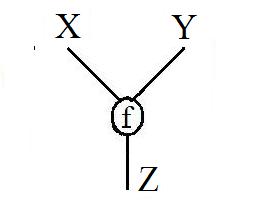}
\end{center}

\noindent The identity morphism on an object $X$ is represented by

\begin{center}  
\includegraphics[scale=0.4]{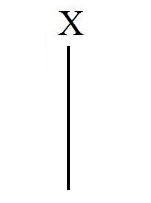}
\end{center}

\noindent The braiding isomorphisms $c$ and $c^{-1}$ are represented
respectively by

\begin{center}  
\includegraphics[scale=0.4]{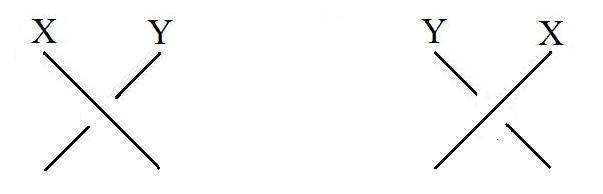}
\end{center}
 
\noindent Composition is by concatenation and tensoring is by juxtaposition.

A framed string diagram besides strings and nodes may have framed regions
labeled by comonoids in $\mathcal{V}$. A framed region labeled by a comonoid $C$ has two
strings passing through it, of which, the left string is labeled by a
right $C$-comodule and the right string is labeled by a left $C$-comodule. Such
a framed region corresponds to taking tensor product over $C$. Now we give the
description.

Suppose that $C$ is a comonoid in $\mathcal{V}$. Suppose that $M$ is a right
$C$-comodule and $N$ is a left $C$-comodule with coactions:

\begin{center}  
\includegraphics[scale=0.4]{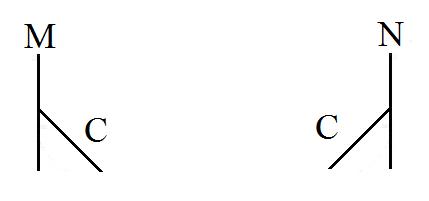}
\end{center}

\noindent The framed string diagram

\begin{center}  
\includegraphics[scale=0.4]{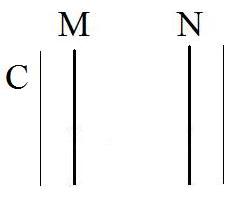}
\end{center}

\noindent represents the identity morphism on $M\otimes_C N$. The framed string
diagram

\begin{center}  
\includegraphics[scale=0.4]{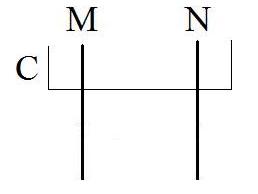}
\end{center}

\noindent represents a canonical injection $i : M\otimes_C N \to M\otimes N$. Observe that we have: 

\begin{center}  
\includegraphics[scale=0.4]{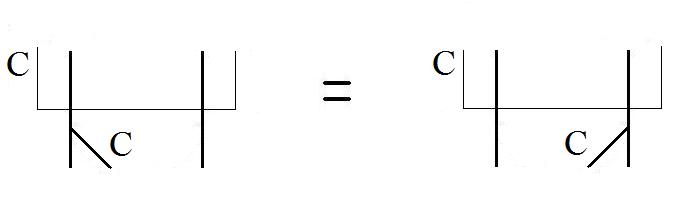}
\end{center}

\noindent The framed string diagram

\begin{center}  
\includegraphics[scale=0.4]{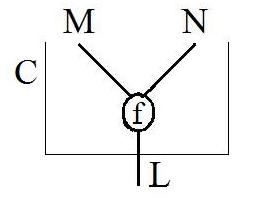}
\end{center}

\noindent represents a morphisms $M\otimes_C N \to L$. The following is a rule
for building a new string diagram from a given one. Suppose that a morphism
$\ldots \to \ldots M\otimes N \ldots$ is represented by a framed string diagram

\begin{center}  
\includegraphics[scale=0.4]{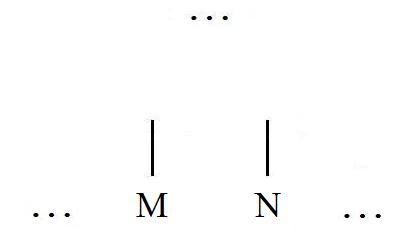}
\end{center}

\noindent and suppose that it factors through a morphism $\ldots \to \ldots
M\otimes_C N \ldots$; that is, the equality 

\begin{center}  
\includegraphics[scale=0.4]{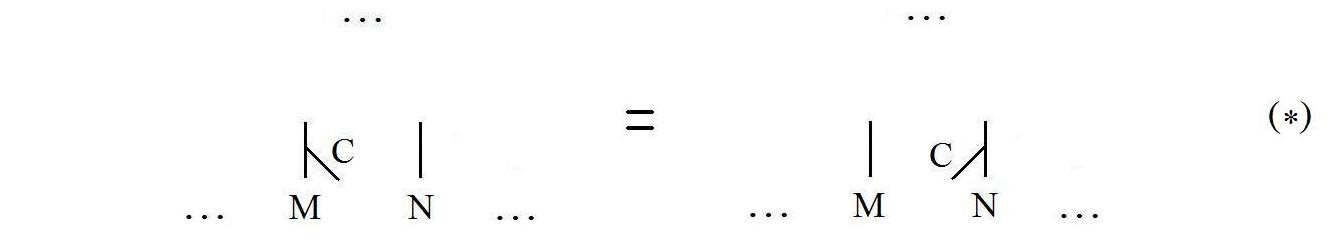}
\end{center}

\noindent holds, then the latter morphism is represented by

\begin{center}  
\includegraphics[scale=0.4]{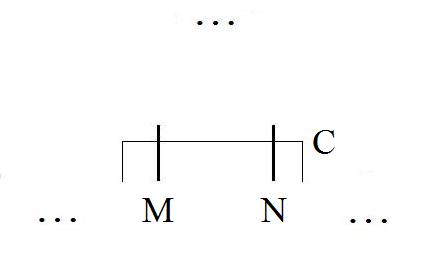}
\end{center}

\noindent Thus every time we want to introduce a new frame in a string diagram
using this rule an extra computation establishing $(\ast)$ should be performed.
We also consider overlapping of framed regions. If $M$ is a right $C$-comodule,
$N$ is a comodule from $C$ to $C'$ and $L$ is a left comodule, then

\begin{center}  
\includegraphics[scale=0.4]{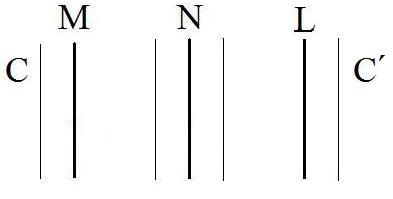}
\end{center}

\noindent represents the identity map on $M\otimes_CN\otimes_{C'}L$.
If $f : N\otimes_{C'}L \to K$ is a left $C$-comodule map and $g :
M\otimes_CN \to P$ is a right $C'$-comodule map, then the framed string diagrams

\begin{center}   
\includegraphics[scale=0.4]{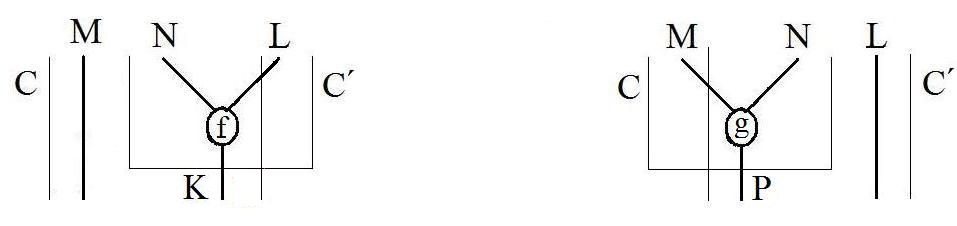}
\end{center}

\noindent represent $1\otimes_C f : M\otimes_CN\otimes_{C'}L \to M\otimes_{C'}K$
and $g\otimes_{C'}1 : M\otimes_CN\otimes_{C'}L \to P\otimes_{C'}L$ respectively.

A quantum graph consists of comonoids $C$ and $A$, for
comultiplications and counits of which we write

\begin{center}  
\includegraphics[scale=0.4]{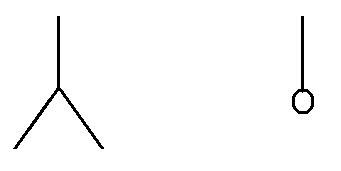}
\end{center}

\noindent and comonoid maps $s : A \to C$, $t : A
\to C$ related by

\begin{center}  
\includegraphics[scale=0.4]{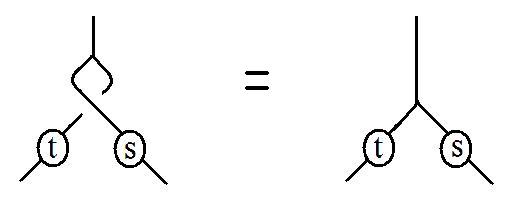}
\end{center}

\noindent A quantum category consists of a quantum graph together with the
composition map $\nu_2 : A\otimes_CA \to A$ and the identity map $\nu_0: C \to A$: 

\begin{center}  
\includegraphics[scale=0.4]{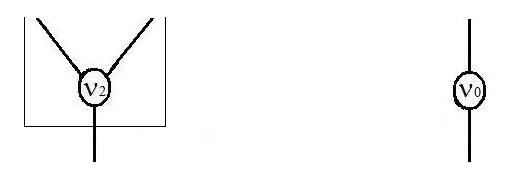}
\end{center}

\noindent which satisfy the six axioms in Statement \ref{statement}. Below we
quickly go through all of these axioms using framed string diagrams.

\noindent $A$ is regarded as a left and right $C$-comodule by coactions:

\begin{center}  
\includegraphics[scale=0.4]{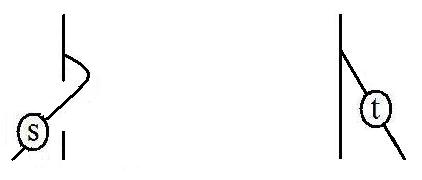}
\end{center}

\noindent The tensor product $H = A\otimes_CA$ of $A$ with itself over $C$ is a
left and a right $C$-comodule by coaction:

\begin{center}  
\includegraphics[scale=0.4]{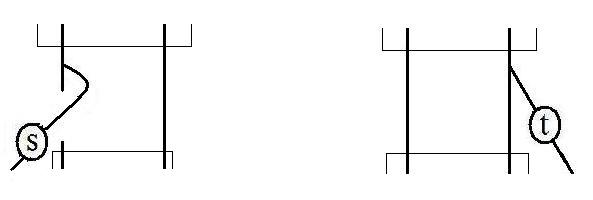}
\end{center}

\noindent Axiom 1 says that $(A, \nu_2, \nu_0)$ should be
a monoid in $\mathrm{Comod}\mathcal{V}(C, C)$. This means that the following
conditions should be satisfied.

\begin{center}  
\includegraphics[scale=0.4]{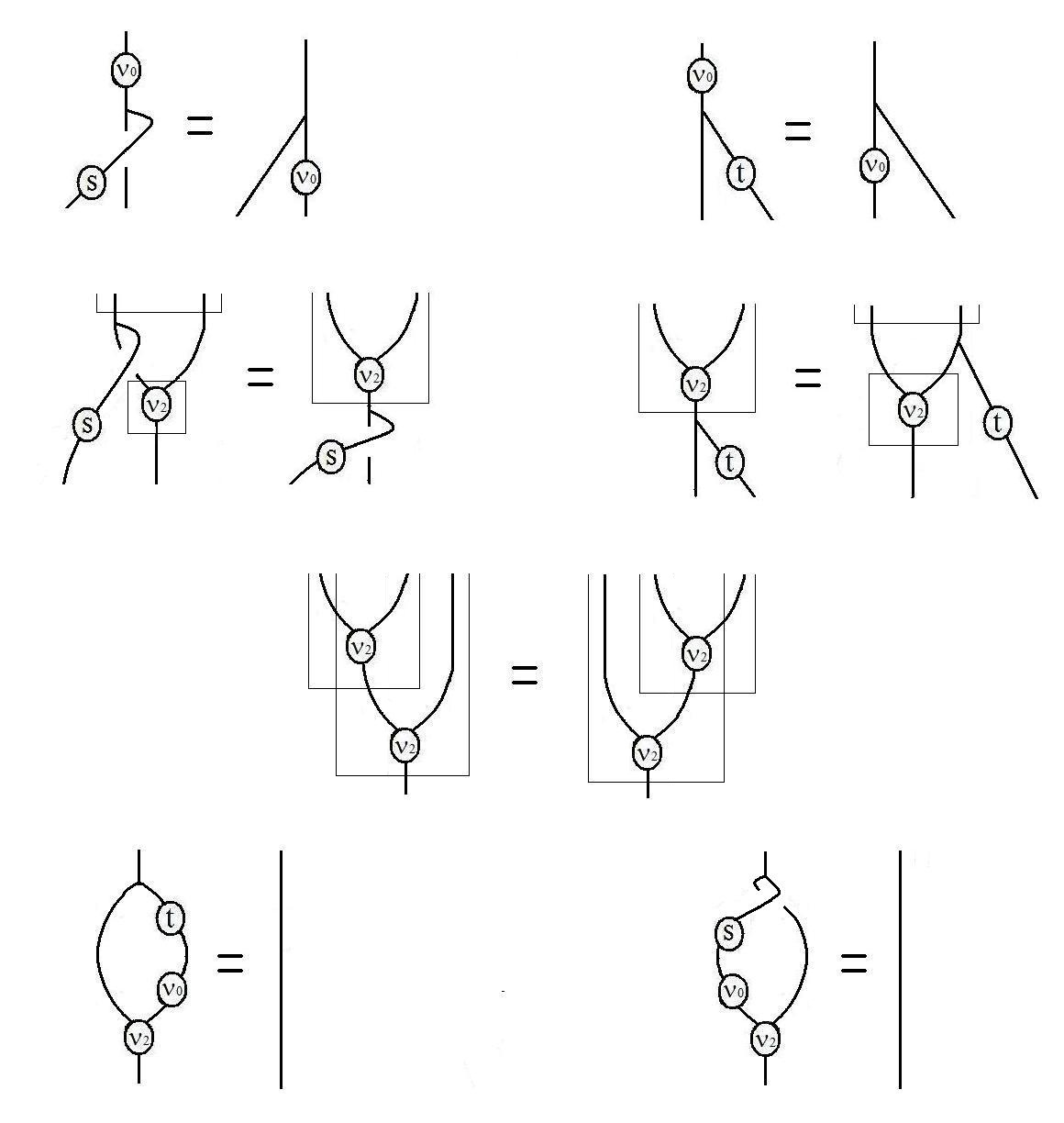}
\end{center}

\noindent We have:

\begin{center}  
\includegraphics[scale=0.4]{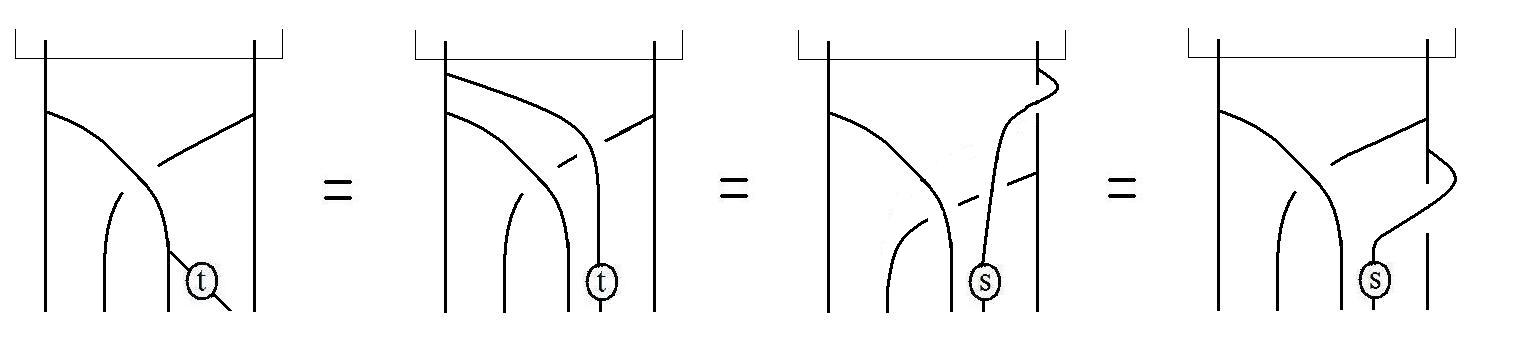}
\end{center}

\noindent Therefore we can form a framed string diagram

\begin{center}  
\includegraphics[scale=0.4]{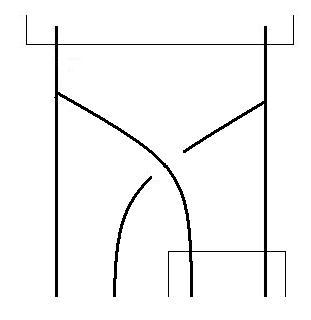}
\end{center}

\noindent This is the map $\gamma_l : H \to A\otimes A\otimes H$. Axiom 2 is:

\begin{center}  
\includegraphics[scale=0.4]{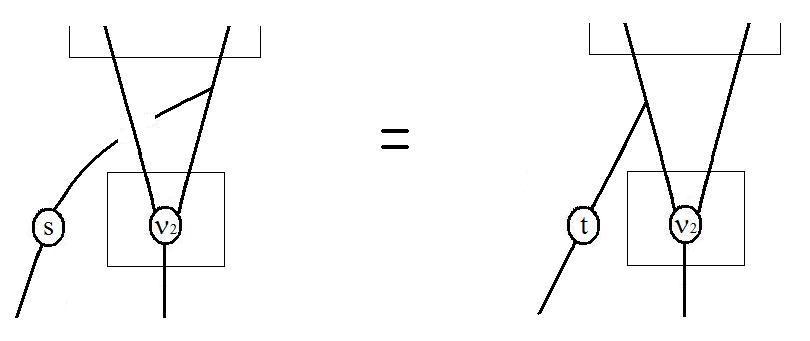}
\end{center}

\newpage

\noindent Using Axiom 2 we have:

\begin{center}  
\includegraphics[scale=0.4]{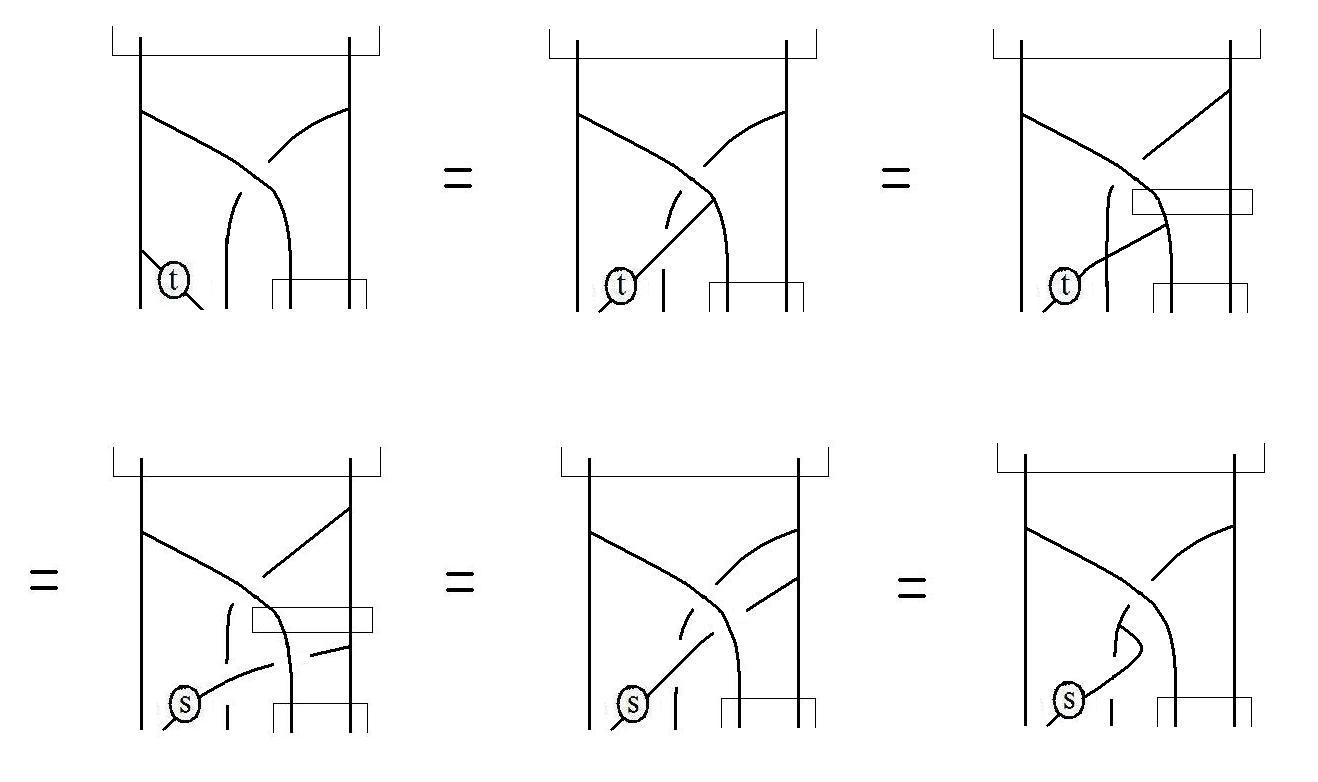}
\end{center}

\noindent Thus, we can form a framed string diagram

\begin{center}  
\includegraphics[scale=0.4]{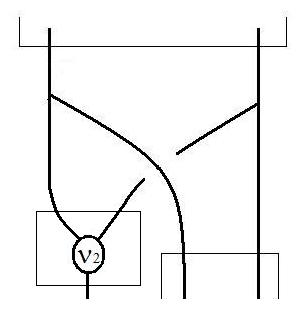}
\end{center} 

\noindent This is the map $\gamma_r : H \to H\otimes A$. Axiom 3:

\begin{center}  
\includegraphics[scale=0.4]{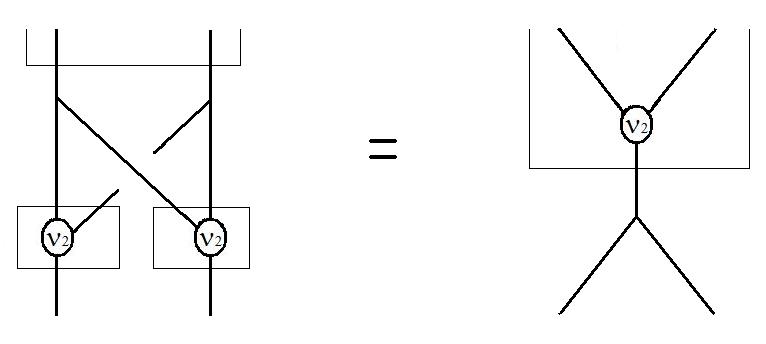}
\end{center}

\newpage

\noindent Axiom 4:

\begin{center}  
\includegraphics[scale=0.4]{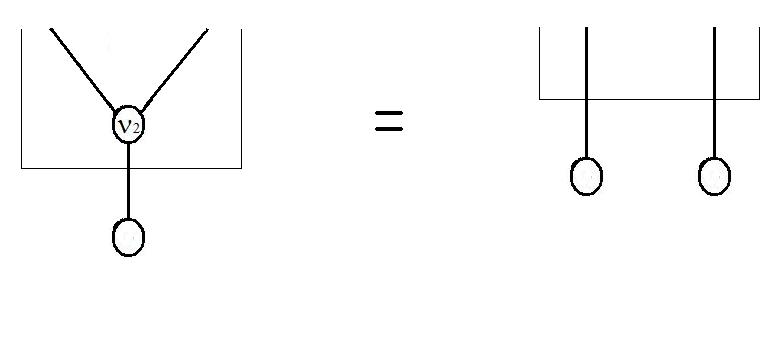}
\end{center}

\noindent Note that these string diagrams look exactly like the string diagrams
for two of the bialgebroid axioms. The map $\gamma_r : C \to C\otimes A$ is
either of:

\begin{center}  
\includegraphics[scale=0.4]{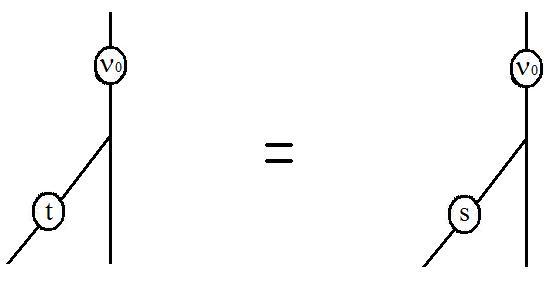}
\end{center}

\noindent Axiom 5:

\begin{center}  
\includegraphics[scale=0.4]{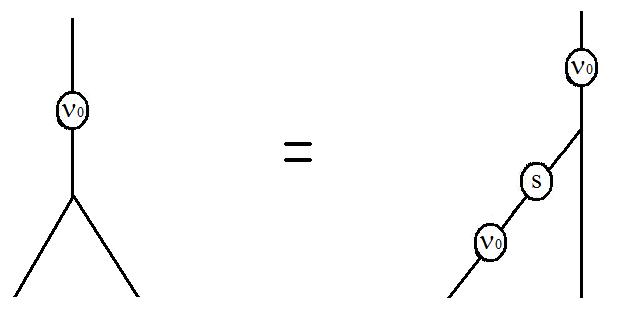}
\end{center}

\noindent Axiom 6:

\begin{center}  
\includegraphics[scale=0.4]{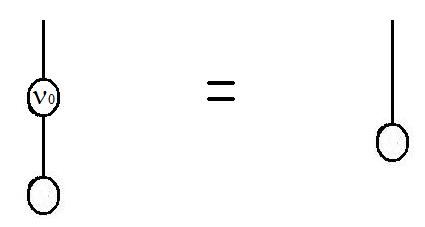}
\end{center}

\newpage

Axioms for a quantum functor $(f, \varphi) : (A, C) \to (A', C')$ in framed
string diagrams are:

\begin{center}  
\includegraphics[scale=0.4]{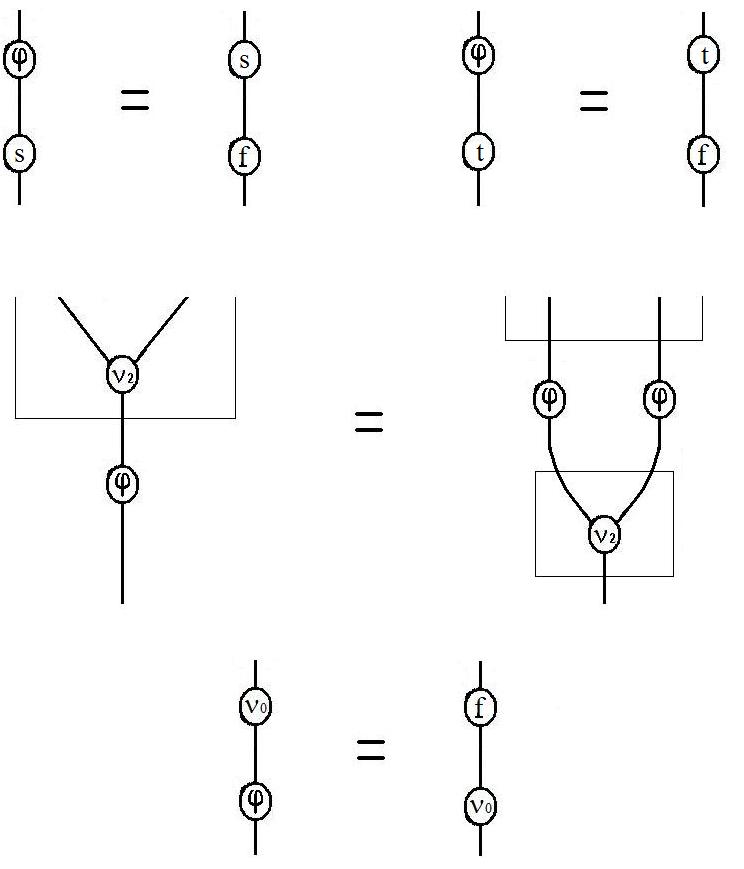}
\end{center}

Axioms for a quantum natural transformation $\tau : (f, \varphi) \to
(g, \varphi')$ in framed string diagrams are:

\begin{center}  
\includegraphics[scale=0.4]{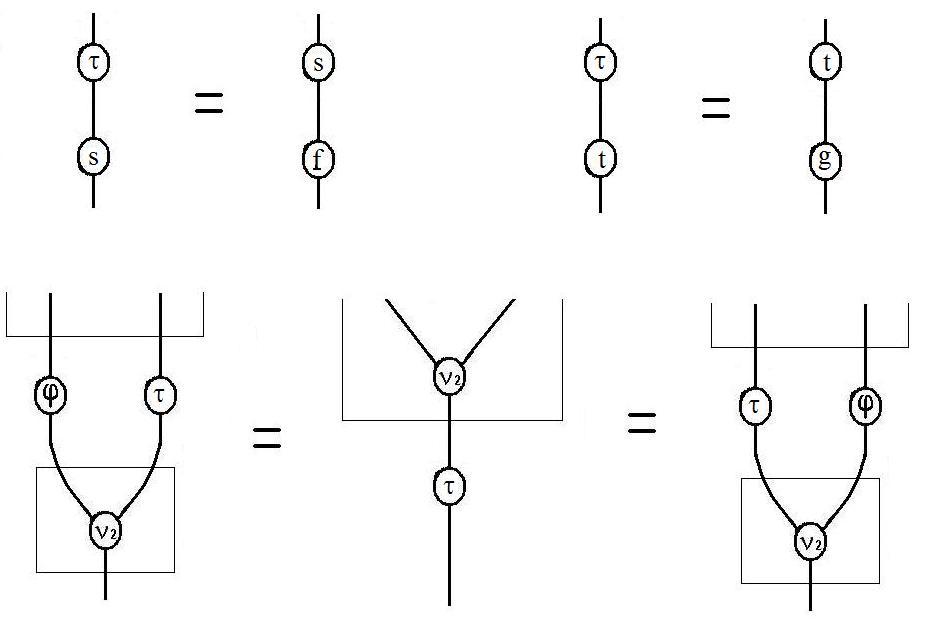}
\end{center}

\end{document}